\documentclass[a4paper,11pt,reqno]{amsart}

\usepackage[utf8]{inputenc}
\usepackage{amsfonts}
\usepackage{amsmath}
\usepackage{amssymb}
\usepackage{amsthm}
\usepackage{comment}
\usepackage{graphicx}
\usepackage{mathtools}
\usepackage{subcaption}
\usepackage{tikz}
\usepackage[margin=1in]{geometry}
\usepackage[numbers]{natbib}
\usepackage[colorlinks=true,linkcolor=blue,urlcolor=blue,citecolor =blue,anchorcolor=blue]{hyperref}

\usepackage[shortlabels]{enumitem}
\usepackage{bbm}

\usepackage{dsfont}

\colorlet{darkgreen}{green!50!black}

\usepackage{hyperref} 
\hypersetup{	
colorlinks=true,  
breaklinks=true,  
urlcolor= green, 
linkcolor= blue,	
citecolor=darkgreen,	
pdftitle={Escape proba}, 
} 

\allowdisplaybreaks
\newcommand*{\de}{\, \mathrm{d}}
\newcommand*{\dd}{\mathrm{d}}

\newcommand*{\D}{\mathrm{d}}

\newcommand*{\ges}{\geqslant}

\newcommand*{\Oh}{\mathrm{O}}

\newcommand*{\mcG}{\mathcal{G}}
\newcommand*{\mcH}{\mathcal{H}}

\newcommand*{\mbC}{\mathbb{C}}

\newcommand{\PP}{{\mathbb P}}

\newcommand{\EE}{{\mathbb E}}

\definecolor{darkgreen}{RGB}{0,100,0}

\newtheorem{theorem}{Theorem}
\newtheorem{lemma}[theorem]{Lemma}
\newtheorem{proposition}[theorem]{Proposition}

\newtheorem{remark}[theorem]{Remark}

\usepackage{soul}
\usepackage[normalem]{ulem}
\newcommand{\rems}[1]{\textcolor{black}{#1}}
\newcommand{\remst}[1]{}
\newcommand{\remstf}[1]{}

\author{Philip A. Ernst \and Sandro Franceschi \and Dongzhou Huang}

\address{Department of Statistics
        Rice University, 6100 Main Street, Houston TX, 77005, USA
        } \email{philip.ernst@rice.edu}
        
        \address{Laboratoire de Mathématiques d’Orsay,
        Université Paris Sud, Bâtiment 307, 91405 Orsay, France
        }
         \address{Department CITI, CNRS UMR 5157 SAMOVAR, T\'el\'ecom SudParis, Institut Polytechnique de Paris, Palaiseau, France
        } \email{sandro.franceschi@universite-paris-saclay.fr}
        
          \address{Department of Statistics
        Rice University, 6100 Main Street, Houston TX, 77005, USA
        } \email{dh37@rice.edu}

\begin{document}
\title[Escape and absorption probabilities for Brownian motion in a quadrant]{Escape and absorption probabilities for obliquely reflected Brownian motion in a quadrant}
%\author[S.\ Franceschi]{Sandro~Franceschi}
%\address{Universit\'e Paris-Saclay}

\begin{abstract}
We consider an obliquely reflected Brownian motion  $Z$  with positive drift in a quadrant stopped at time $T$, where $T:=\inf \{ t>0 : Z(t)=(0,0) \}$ is the first hitting time of the origin. Such a process can be defined even in the non-standard case where the reflection matrix is not completely-$\mathcal{S}$. We show that in this case the process has two possible behaviors: either it tends to infinity or it hits the corner (origin) in a finite time. Given an arbitrary starting point $(u,v)$ in the quadrant, we consider the escape (resp. absorption) probabilities $\mathbb{P}_{(u,v)}[T=\infty]$ (resp. $\mathbb{P}_{(u,v)}[T<\infty]$). We establish the partial differential equations and the oblique Neumann boundary conditions which characterize the escape probability and provide a functional equation satisfied by the Laplace transform of the escape probability. We then give asymptotics for the absorption probability in the simpler case where the starting point in the quadrant is $(u,0)$. We exhibit a remarkable geometric condition on the parameters which characterizes the case where the absorption probability has a product form and is exponential. We call this new criterion the \textit{dual skew symmetry} condition due to its natural connection with the skew symmetry condition for the stationary distribution. We then obtain an explicit integral expression for the Laplace transform of the escape probability. We conclude by presenting exact asymptotics for the escape probability at the origin. 
%Furthermore, we identify some very interesting cases where the boundary problem can be solve thanks to Tutte's invariant which provides us a simple explicit integral free expression for the absorption probability.
\end{abstract}

\keywords{Escape probability; Absorption probability; Obliquely reflected Brownian motion in a quadrant; Functional equation; Carleman boundary value problem; 
%Tutte's invariant ;  
Laplace transform; Neumann's condition; Asymptotics}

\maketitle
%%%%%%%%%%%%%%%%%%%%%%%%%%%%%%%%%%%%%%%%%%%%%%%%%%%%%%%%%%%%%%%%%%%%%%

%\tableofcontents

\section{Introduction}

%\rems{test} \remst{test} \remstf{test}

\subsection{Model and goal}

Let $Z(t)=(Z_1(t),Z_2(t))$ be a reflected Brownian motion (RBM) in the quadrant, starting from the point $(u,v)$, with positive drift $\mu=(\mu_1,\mu_2) $; that is,
$
\mu_1>0, \mu_2 >0.
$ 
The covariance matrix is \rems{ 
$\begin{psmallmatrix}1&\rho\\ \rho&1\end{psmallmatrix}$
with $|\rho|<1$}
and the reflection matrix is 
\rems{$\begin{psmallmatrix}1&-r_2\\ -r_1&1\end{psmallmatrix}$}. 
We further assume that 
\begin{equation}
r_1>0, \ r_2>0 \text{ and } 1 \leqslant r_1r_2.
\label{eq:conditionRmatrix}
\end{equation}
See Figure~\ref{fig:rebondrevice} for a representation of the parameters. We define this reflected process up to the first hitting time $T$ of the corner, defined as
$$
T:= \inf \{ t>0 : Z(t)=0 \}.
$$
%is not? a semimartingale and 
For $t\leqslant T$, this process may be written as
\begin{equation} \label{ourprocess}
\begin{cases}
Z_1(t):=u+W_1(t)+\mu_1 t+ l_1(t) -r_2 l_2(t),
\\
Z_2(t):=v+W_2(t)+\mu_2 t - r_1 l_1(t) +l_2(t),
\end{cases}
\end{equation}
 where
$l_1(t)$ (resp. $l_2(t)$) is a local time on the vertical (resp. horizontal) axis and is a continuous non-decreasing process \rems{starting from $0$} which increases only
when $Z_1 (t) =0$ (resp. $Z_2 (t) =0$). Under condition $\eqref{eq:conditionRmatrix}$, \rems{when $t>T$, that is after that the process $Z$ hits the corner, the process is no longer defined by \eqref{ourprocess} for reasons of convexity. In lieu, for $t>T$, we define $Z(t)=(0,0)$ and say that the process is absorbed when $T<\infty$}. Further details on the existence and uniqueness of this process will be given in Section~\ref{subsec:definitionprocess}.

The objective of the present paper is to study the \rems{probability of escape} to infinity for a process starting from $(u,v)$. We denote this probability as $$\mathbb{P}_{(u,v)}[T=\infty].$$
The corresponding absorption probability at the origin is
$\mathbb{P}_{(u,v)}[T<\infty]=1 -\mathbb{P}_{(u,v)}[T=\infty].$

\begin{figure}[hbtp]
\centering
\includegraphics[trim={0cm 0.6cm 0cm 0.8cm}, clip,scale=0.5]{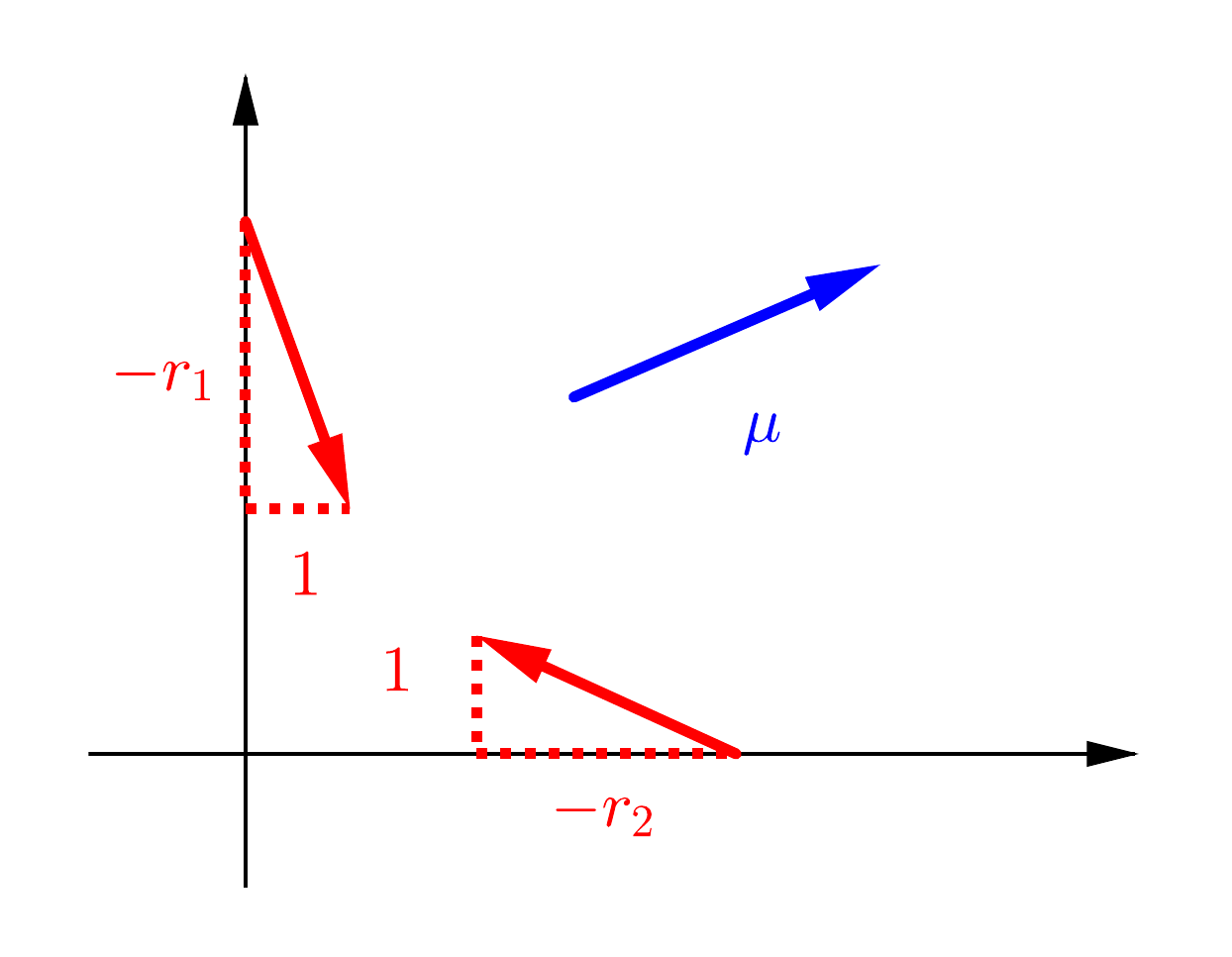}
\caption{Reflection vectors and drift.}
\label{fig:rebondrevice}
\end{figure}

%\rems{referee suggestion p1: could you add a few lines of motivation explaining the interest in looking at this specific problem of the espace probability? I don't know if this is relevant, but in some papers such as
%    Ratzkin, Jesse; Treibergs, Andrejs
%    A capture problem in Brownian motion and eigenvalues of spherical domains.
%    Trans. Amer. Math. Soc. 361 (2009), no. 1, 391-405 some quantities related to a Brownian pursuit are considered. This may be a relevant motivation, or others related to finance?}
%I don't think the reference given by the referee is appropriate but I have cited other references bellow and I detailed some possible motivations (I also cite Jev Ivanovs a bit more).

%Obliquely reflected Brownian motion has been linked to the two-dimensional version of Lévy's theorem about the maximum process of linear Brownian motion \cite{LeGall1987}. 

\indent Since its introduction in the eighties by Harrison, Reiman, Varadhan and Williams \cite{harrison_reiman_1981,harrison_reiman,williams_recurrence_1985,
varadhan_brownian_1985,williams_1985_reflected}\rems{,}
reflected Brownian motion in the quarter plane has received significant attention from
probabilists. Recurrence and transience of obliquely \rems{reflected} Brownian motion were examined in \cite{hobson_rogers,williams_recurrence_1985}, and the process has also been considered in planar domains \cite{harrison_landay_shepp_stationary_1985, harrison_shepp_tandem_84} as well as in general dimensions in orthants \cite{harrison_reiman,taylor_existence_1993,williams_survey}. The stationary distribution of obliquely \rems{reflected} Brownian motion has been studied in \cite{dai_reflecting_2011,dieker_reflected_2009,
franceschi_2019,kang_2014,dupuis_2014_timereverse} and its Green's functions have been studied in \cite{franceschi_green_2020}. \rems{The roughness of its paths has been studied in \cite{LRZ}.} Obliquely \rems{reflected} Brownian motion  has played an important role in applications concerning heavy traffic approximations for open queueing networks (\cite{harrison_1978,reiman_84_open}). It has also been utilized in queueing models as diffusion approximations for tandem queues (\cite{lieshout_tandem_2007,lieshout_asymptotic_2008, miyazawa_tail_2009}). 

\rems{There are several possible interpretations in insurance risk of models involving reflected Lévy processes in a quadrant~(\cite{albrecher_azcue_muler_2017,BBRW,Ivanovs-Boxma-bivariat-2015}). For example, suppose there are two funds, each of whose free surplus is modelled by a Cram\'er-Lundberg process, and which have the following agreement: a deficit in one fund is instantly covered by the other fund, with ruin occuring when neither company can cover the deficit of the other. In the case of our problem, the absorption probability could be interpreted to be the probability of ruin; the escape probability may be interpreted as the probability of survival and infinite capital expansion. The aforementioned process also arises in the study of queueing models as diffusion approximations
for some Lévy tandem queues~(\cite{boxma2013,fomichov2020probability,Whitt2002}).}

\indent Previous works (\cite{baccelli_analysis_1987,
ernst_2020,foddy_analysis_1984,franceschi_2019,fomichov2020probability}) have adapted an analytic method initially developed for random walks by \citet{fayolle_two_1979} and \citet{malysev_analytic_1972} for studying obliquely reflected Brownian motion. \rems{Above all we mention \cite{fomichov2020probability} which focus on a non-standard regime where the reflected process escapes to
infinity along one of the axes. The techniques and the results employed to solve this problem are very similar to this article.} This method is based on the boundary value problem theory documented by the books of \citet{fayolle_random_2017} and \citet{cohen_boxma}. The present article is in part inspired by this analytic approach.\\

\subsection{Definition of the process given in (\ref{ourprocess})}
\label{subsec:definitionprocess}

Brownian motion in a quadrant with oblique reflection is usually defined as a process which behaves as a standard Brownian motion in \rems{the interior} of the quadrant,\remst{ The process} reflects instantaneously on the edges with constant direction and the amount of time spent at the origin has Lebesgue measure zero (\citet{varadhan_brownian_1985}). Such a process is defined as a solution of a submartingale problem \cite{varadhan_brownian_1985}. An interesting case arises when the process is a semimartingale reflecting Brownian motion (SRBM).
\citet{reiman_boundary_1988} showed that a necessary condition for the process to be \rems{an} SRBM is for the reflection matrix to be completly-$\mathcal{S}$\footnote{
A square matrix $R$ is said to be completly-$\mathcal{S}$ if for each principal sub-matrix $\widetilde{R}$ there exists $\widetilde{x} \geqslant 0$ such that $\widetilde{R} \widetilde{x} >0$.}.
\citet{taylor_existence_1993} showed that this condition was also sufficient for the existence of an SRBM, which is unique in law.

Due to condition ~\eqref{eq:conditionRmatrix}, the reflection matrix of the process in \eqref{ourprocess} is not completely-$\mathcal{S}$. The process indeed is not a standard SRBM as it may be trapped at the origin.
Nonetheless, it is possible to define this absorbed process up \rems{to} the stopping time $T$. The existence and uniqueness as a solution of a submartingale problem \remst{is given }for the absorbed process is given in \cite[\S 2.1, Thm 2.1]{varadhan_brownian_1985}. Further, in \citet[\S 4.2 and \S 4.3]{taylor_existence_1993}, the existence and uniqueness of an SRBM absorbed at the origin are proven without assuming that the reflection matrix is completely-$\mathcal{S}$.

\subsection{From the quadrant to the wedge}

 \citet[Appendix]{franceschi_2019} recently showed that studying reflected Brownian motion in a quadrant is equivalent to studying reflected Brownian motion in a wedge with angle $\beta$, with identity covariance matrix, with two reflection angles $\delta$ and $\epsilon$, and with drift angle $\theta$ (see Figure~\ref{fig:wedgeangle}).
The angles $\delta, \epsilon, \beta$ and $\theta$ (when the drift is nonzero) are in $(0,\pi)$ and are defined by
\begin{equation}
\tan\delta= \frac{\sin \beta}{-r_2 +\cos\beta},
\quad
\tan\epsilon= \frac{\sin \beta}{-r_1 +\cos\beta},
\quad
\tan\theta= \frac{\sin \beta}{\mu_1/\mu_2 +\cos\beta},
\quad
\cos\beta= -\rho.
\label{eq:defangles}
\end{equation}
\rems{The angles are equal to $\pi/2$ when the denominators of the tangents are equal to $0$.}

\vspace{-0.3cm}
\begin{figure}[hbtp]
\centering
\includegraphics[trim={0cm 0.5cm 0cm 0.7cm}, clip,scale=0.8]{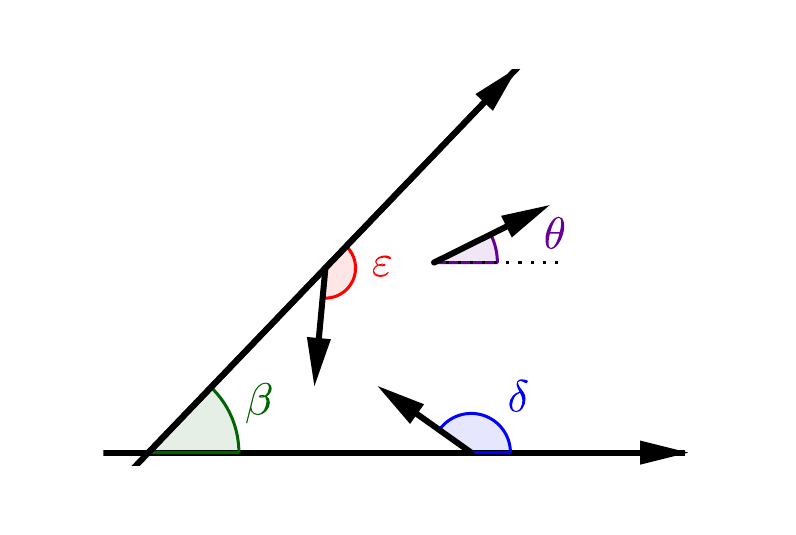}
\caption{Reflected Brownian motion \remst{with zero drift }in a wedge of angle $\beta$, reflection angles $\delta$ and $\epsilon$, and drift angle $\theta$.}
\label{fig:wedgeangle}
\end{figure}

\noindent Finally, we denote $\alpha$, now a standard quantity in the SRBM literature, to be 
\begin{equation}
\alpha:= \frac{\delta+\epsilon-\pi}{\beta}. 
\label{eq:alpha}
\end{equation}
Condition \eqref{eq:conditionRmatrix} is equivalent to $\delta+\epsilon-\beta\geqslant \pi$ (or equivalently $\alpha\geqslant1$) and $\delta>\beta$, $\epsilon>\beta$. 

\subsection{The case of zero drift}

The case of zero drift $\mu=0$ was treated by \citet{varadhan_brownian_1985}.
In this case the absorption probability does not depend on the starting point. We have from \cite[Thm 2.2]{varadhan_brownian_1985}
$$
\mathbb{P}[T<\infty]=
\begin{cases}
1 & \text{if } \alpha>0,
\\ 0 & \text{if } \alpha \leqslant 0 .
\end{cases}
$$
If $\alpha \leqslant 0$, the corner is not reached. If $0< \alpha <2$, the corner is reached but the amount of time spent by the process in the corner is Lebesgue measure zero. If $\alpha \geqslant 2$, the process reaches the corner and remains there. \rems{The previous properties are valid in the case of zero drift.} 
Under condition \eqref{eq:conditionRmatrix}, the case of positive drift poses a \rems{new challenge}, as $0<\mathbb{P}_{(u,v)}[T<\infty]<1.$ \rems{Remark that} condition ~\eqref{eq:conditionRmatrix} is equivalent to $\alpha \geqslant 1$.
\subsection{Escape probability and stationary distribution of the dual process}
\label{subsec:dualprocess}

\citet{harrison_1978} and \citet{foddy_analysis_1984}
showed that the hitting time on one of the axes is intrinsically connected to the stationary distribution of a certain dual process. As the present article was nearing completion, it came to our attention that \citet{harrison_2020} has extended the results in his earlier work (\cite{harrison_1978}) by introducing a dual RBM in the quadrant with drift $-\mu$ and reflection matrix $$\left(\begin{array}{cc}
r_2 & -1 \\ 
-1 & r_1
\end{array}\right)$$ 
when $1<r_1r_2$.
This \rems{configuration of parameters} is depicted in Figure~\ref{fig:dual} below. This dual process has an explicit connection with the study of the escape probability. In particular, \rems{\citet[Cor. 2]{harrison_2020}} states that
$$
\mathbb{P}_{(u,v)}[T=\infty] = \pi (\mathcal{S}(u,v))\rems{,}
$$
where $\pi$ is the stationary distribution of the dual process and $\mathcal{S}(u,v):=\{(u-r_2 z_1+z_2, v+z_1-r_1z_2)\in\mathbb{R}_+^2 : (z_1,z_2)\in\mathbb{R}_+^2 \}$ is a trapezoid as pictured in Figure~\ref{fig:dual}.
\begin{figure}[hbtp]
\centering
\includegraphics[trim={0cm 0.6cm 0cm 0.8cm}, clip,scale=0.6]{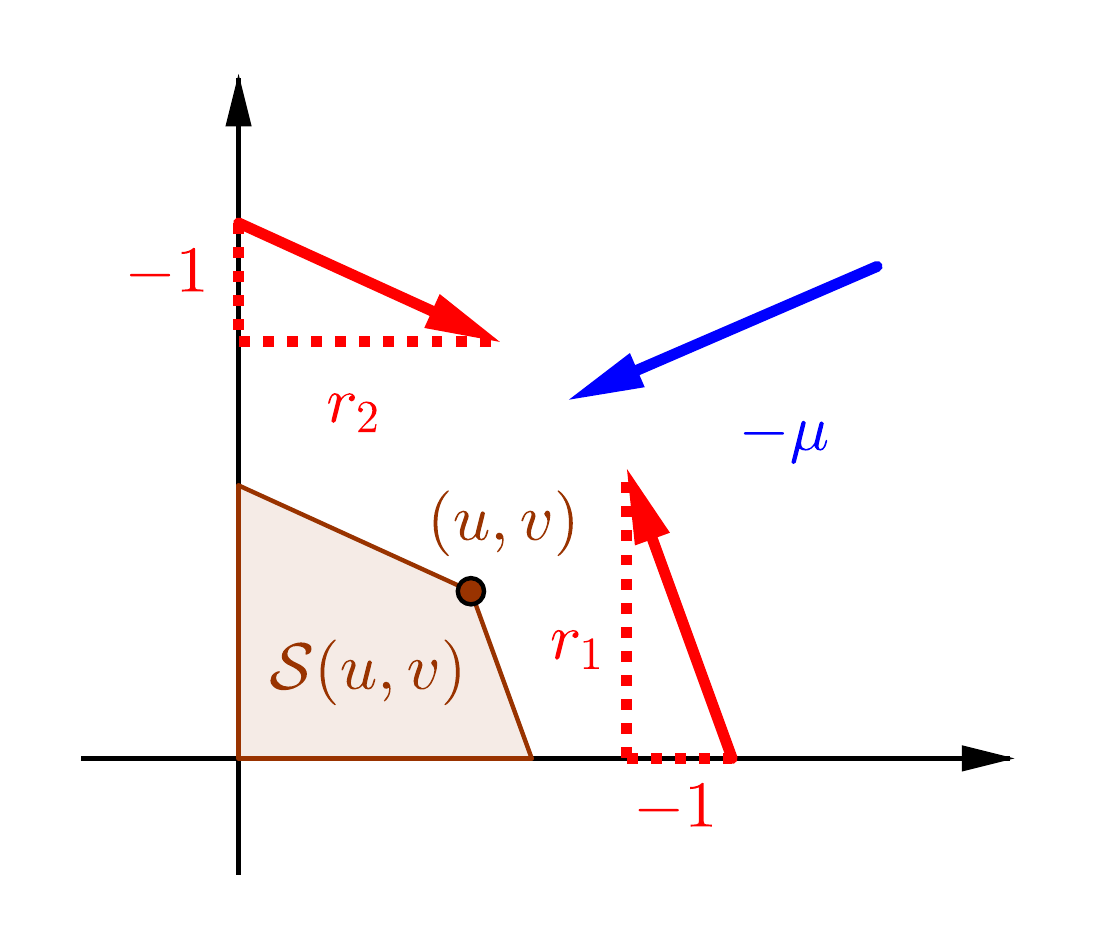}
\caption{Dual process parameters and trapezoid $\mathcal{S}(u,v)$ in brown.}
\label{fig:dual}
\end{figure}

\subsection{Plan}

\indent The remainder of this paper is organized as follow\rems{s}. In Section~\ref{sec:behaviour} we explore some general properties of the process of interest given in (\ref{ourprocess}). This section's key result is Theorem \ref{thm:Ztoinfty}, which states that the process has only two possible behaviors: either $T<\infty$, which means that the process is absorbed at the origin in finite time, or $T=\infty$, in which case the process escapes to infinity, namely $Z(t)\to\infty$ when $t\to\infty$\rems{.}
In Section~\ref{sec:functionalequation} we present Proposition~\ref{prop:PDE}, which gives a partial differential equation characterizing the escape probability. Later in this section, we give Proposition~\ref{prop:functionaleq}, which provides a functional equation satisfied by the Laplace transform of the escape probability. In Section~\ref{sec:kernelasympt}, we study the kernel of this functional equation and obtain asymptotic\remst{s} results for the absorption probability in the simpler case where the starting point is $(u,0)$ (Proposition~\ref{cor:asymptBrown}). In Section~\ref{sec:productform} we find a geometric condition which characterize\rems{s} the cases where the absorption probability has a product form and is exponential (Theorem~\ref{thm:productform}). Such a result recall\rems{s} the famous skew symmetry condition studied a lot for invariant measures (\cite{harrison_reiman_1981,harrison_multidimensional_1987}). In Section~\ref{sec:integralexpression} we establish a boundary value problem (BVP) satisfied by the Laplace transform \rems{of the escape probability} (Proposition~\ref{prop:BVPBrownian}) and conclude with Theorem~\ref{thm:explicitF1Brown}, which gives an explicit integral formula \rems{of this transform}\remst{for the Laplace transform of the escape probability}. In Section~\ref{sec:asympt0} we obtain exact asymptotics for the escape probability at the origin. \\

%\citet{albrecher_17,ivanovs_boxma}

%Tutte's invariant method \cite{}

 %\citet{kourkova_random_2011,kurkova_martin_1998}  

%\subsection{Results}

\noindent \textbf{In memory of Larry Shepp\rems{.}} We dedicate this article in memory of our colleague, mentor, and friend, Professor Larry Shepp. Professor Shepp indelibly contributed to many areas of applied probability, and one of the areas that interested him most concerned RBM in a quadrant as well \rems{as} in a strip (\cite{harrison_landay_shepp_stationary_1985,harrison_shepp_tandem_84}).

%Finally, in Section~\ref{sec:remarkablecase} we exhibit some very interesting cases where the BVP can be solved thanks decoupling functions and Tutte's invariant, and we state explicit simple formulas.

\section{General properties \rems{of process $Z$}}
\label{sec:behaviour}

In this section we investigate a few key properties of the process given in (\ref{ourprocess}). We prove three key results. The first \rems{one} is that if the starting point tends to infinity, then the probability that the process does not hit the origin tends towards $1$  (Theorem~\ref{thm:uvtendtoinfty}). The second \rems{one} is that when the starting point tends to the origin, the probability that the process hits the origin in finite time tends towards $1$ (Theorem~\ref{thm:lim0}). The third key result is that the process has only two possible behaviors: either $T<\infty$, which means that the process is absorbed at the origin in finite time, or $T=\infty$, in which case the process escapes to infinity, namely $Z(t)\to\infty$ when $t\to\infty$ (Theorem \ref{thm:Ztoinfty}).

%In this section, we want to show $\PP_{(u,v)} [T_{1/n} = \infty \text{ for some n $\in$ $\mathbb{N}_{+}$}  ] \rightarrow 1$ as $(u,v) \rightarrow \infty$. 

\subsection{{\rems{Limits of the hitting probability}}}
\label{sec21}
Our first key results of the section (Theorems \ref{thm:uvtendtoinfty} and \ref{thm:lim0}) concern the probability of the process hitting the origin. \rems{We wish to show that $\lim_{\rems{\|(u,v)\|} \rightarrow \infty} \, \PP_{(u,v)} [T = \infty ] =1$}.  We shall prove this with the aid of Lemma \ref{lemma:crucial1} and Proposition \ref{proposition:hittingtime}.\\
\indent For ease of notation, let us define
$\tau_{1}^{\xi} := \inf \{ t: Z_{1}(t \wedge T) \leq \xi \}$ and
$\tau_{2}^{\xi} := \inf \{ t: Z_{2}(t \wedge T) \leq \xi \}$.
Further, let $X_{1}(t)$:= $u + W_{1}(t) + \mu_1 t$ and let $X_{2}(t)$:=$v + W_{2}(t) + \mu_2 t$.

\rems{Suppose $Z(t)$ is a one-dimensional reflected Brownian motion. The analysis of $Z(t)$ is converted to that of one-dimensional Brownian motion with a drift by the Skorokhod map. However, in the case of obliquely reflected Brownian motion in a quadrant, this method does not generally work due to the presence of $l_1(t)$ and $l_2(t)$. However, on the event $\{\tau_{1}^{\xi} = \infty\}$, note that $l_{1}(t)=0$, the previously reflected Brownian motion becomes an obliquely reflected Brownian motion in a half-plane. This allows the one-dimensional techniques to be applied in our case. This motivates us to consider the event $\{\tau_{1}^{\xi}\}$ below.}
\begin{lemma} \label{lemma:crucial1}
For $u> \xi >0$, we have
\begin{equation}
\PP_{(u,v)} [\tau_{1}^{\xi} = \infty]
= \PP_{(u,v)} \left[X_{1}(t \wedge T) - r_2 \, \sup_{0 \leq s \leq t\wedge T} (-X_{2}(s))^{+} > \xi \text{ for every $t \geq 0$}\right], \label{eq:tau=infty}
\end{equation}
where $x^{+}$ equals $x$ if $x>0$ and is $0$ otherwise. Hence,
\begin{equation}
\PP_{(u,v)} [\tau_{1}^{\xi} = \infty]
\geq \PP_{(u,v)} \left[X_{1}(t) - r_2 \, \sup_{0 \leq s \leq t} (-X_{2}(s))^{+} > \xi \text{ for every $t \geq 0$}\right].   \label{eq:estimateoftau=infty}
\end{equation}
A symmetrical result holds for $v> \xi >0$ and
$
\PP_{(u,v)} [\tau_{2}^{\xi} = \infty]$.
\end{lemma}
\begin{proof}
On the event $\{ \tau_{1}^{\xi} = \infty \}$,  for every $t \geq 0$, we have $l_{1}(t) =0$, $\PP_{(u,v)}$-a.s. Then
\begin{eqnarray*}
 && Z_{1} (t \wedge T) = X_{1} (t \wedge T) - r_2 \, l_{2}(t \wedge T), \\
 && Z_{2} (t \wedge T) = X_{2} (t \wedge T) + l_{2}(t \wedge T).
\end{eqnarray*}
Note that $l_{2}(t \wedge T)$ increases only when $Z_{2} (t \wedge T) =0$. \rems{By uniqueness of the Skorokhod map (see e.g. \cite{Pilipenko} and references therein)}
\begin{equation*}
l_{2} (t \wedge T) = \sup_{0 \leq s \leq t} (- X_{2}(s \wedge T))^{+}
= \sup_{0 \leq s \leq t \wedge T} (- X_{2}(s ))^{+}.
\end{equation*}
Thus
\begin{equation*}
 Z_{1} (t \wedge T) = X_{1}(t \wedge T) - r_2 \, \sup_{0\leq s \leq t\wedge T} (-X_{2}(s))^{+}.
\end{equation*}
We may then write
\begin{eqnarray}
&&\{ \tau_{1}^{\xi} = \infty \} 
= \{ Z_{1}(t \wedge T) > \xi \text{ for every $t \geq 0$} \} \notag \\
&=& \{ Z_{1}(t \wedge T) > \xi \text{ for every $t \geq 0$ and $l_1(T)=0$}  \}  \notag \\
&=& \left\{ X_{1}(t \wedge T) - r_2 \, \sup_{0\leq s \leq t\wedge T} (-X_{2}(s))^{+} > \xi \text{ for every $t \geq 0$ and $l_1(T)=0$} \right \}, \label{eq:1.1}
\end{eqnarray}
 $\PP_{(u,v)}$-a.s. We now wish to show that
\begin{equation}
\PP_{(u,v)} \left[ X_{1}(t \wedge T) - r_2 \, \sup_{0\leq s \leq t\wedge T} (-X_{2}(s))^{+} > \xi \text{ for every $t \geq 0$ and $l_1(T)>0$} \right] =0. \label{eq:1.2}
\end{equation}
Note that there is a set $N$ such that $\PP_{(u,v)} (N) =1 $ and for every $\omega \in N$, we have
\begin{eqnarray}
&& Z_{1} (t \wedge T) = X_{1}(t \wedge T) + l_{1}(t \wedge T) - r_2 \, l_{2}(t \wedge T) \geq 0, \label{eq:1.3} \\
&& Z_{2} (t \wedge T) = X_{2}(t \wedge T) - r_1 \, l_{1}(t \wedge T) + l_{2}(t \wedge T) \geq 0, \label{eq:1.4} \\
&& \text{$l_{1}(t \wedge T)$ increases only when $Z_{1}(t \wedge T)=0$}, \label{eq:1.5} \\
&& \text{$l_{2}(t \wedge T)$ increases only when $Z_{2}(t \wedge T)=0$}. \label{eq:1.6}
\end{eqnarray}
Let \remst{us }$\omega \in N$. We claim that the following statements 
\begin{enumerate}[a)]
\item $X_{1}(t \wedge T) - r_2 \, \sup_{0\leq s \leq t\wedge T} (-X_{2}(s))^{+} > \xi$ for every $t \geq 0$;
\item $l_{1}(T) >0 $,
\end{enumerate}
cannot hold simultaneously. The proof is by contradiction. For sake of contradiction, assume that statements a) and b) hold simultaneously. By \eqref{eq:1.4}, \eqref{eq:1.6}, and the uniqueness of Skorokhod map, we have
\begin{eqnarray*}
&&l_2 (t \wedge T) = \sup_{0 \leq s \leq t} (r_1 \, l_{1} (s \wedge T) - X_{2}(s \wedge T))^{+} \\
&\leq& \sup_{0 \leq s \leq t} (r_1 \, l_{1}(s\wedge T))^{+} + \sup_{0 \leq s \leq t} (- X_{2}(s \wedge T))^{+} \\
&=& r_1 \, l_{1}(t \wedge T) +  \sup_{0 \leq s \leq t \wedge T} (- X_{2}(s ))^{+}.
\end{eqnarray*}
Let $\eta := \inf \{ t : l_{1}(t \wedge T) \geq \xi/(2 r_1 r_2) \}$. Then for every $t \geq 0$,

\rems{\begin{eqnarray*}
&&Z_{1} (t \wedge \eta \wedge T) = X_1(t \wedge \eta \wedge T) + l_1(t \wedge \eta \wedge T) - r_2 l_2(t \wedge \eta \wedge T) \\
&\geq& X_{1}(t \wedge \eta \wedge T) - r_2 \, l_{2} (t \wedge \eta \wedge T) \\
&\geq& X_{1}(t \wedge \eta \wedge T) - r_2 \, \sup_{0 \leq s \leq t \wedge \eta \wedge T} (- X_{2}(s ))^{+}
 - r_1 r_2 \, l_{1}(t \wedge \eta \wedge T)  \\
&>& \xi - \xi/2 = \xi/2,
\end{eqnarray*}}
where in the last inequality we have invoked statement a). Since $l_1(t \wedge T)$ increases only when $Z_{1}(t \wedge T) =0$, we have
\begin{equation*}
l_{1}(t\wedge \eta \wedge T)=0 \text{ for every $t\geq 0$},
\end{equation*}
which contradicts statement b) and the definition of $\eta$. By contradiction, \eqref{eq:1.2} holds. Combining \eqref{eq:1.1} and \eqref{eq:1.2}, \eqref{eq:tau=infty} follows. Note that \eqref{eq:estimateoftau=infty} follows directly from \eqref{eq:tau=infty}.
\end{proof}

%\begin{lemma} \label{lemma:crucial2}
%For $v> \xi >0$, we have
%\begin{equation}
%\PP_{(u,v)} [\tau_{2}^{\xi} = \infty]
%= \PP_{(u,v)} \left[X_{2}(t \wedge T) - r_1 \, \sup_{0 \leq s \leq t\wedge T} (-X_{1}(s))^{+} > \xi \text{ for every $t \geq 0$}\right]. \label{eq:tau=infty2}
%\end{equation}
%Hence,
%\begin{equation}
%\PP_{(u,v)} [\tau_{2}^{\xi} = \infty]
%\geq \PP_{(u,v)} \left[X_{2}(t) - r_1 \, \sup_{0 \leq s \leq t} (-X_{1}(s))^{+} > \xi \text{ for every $t \geq 0$}\right].   \label{eq:estimateoftau=infty2}
%\end{equation}
%\end{lemma}
%Since the proof of Lemma \ref{lemma:crucial2} is nearly identical to that of Lemma \ref{lemma:crucial1}, we omit the proof.

%We now turn to Proposition \ref{proposition:hittingtime} below.

\rems{
\begin{remark}
To estimate the probability of the event
\begin{equation*}
\{X_{1}(t) - r_2 \, \sup_{0 \leq s \leq t} (-X_{2}(s))^{+} > \xi \text{ for every $t \geq 0$}\},
\end{equation*}
we note that the above event contains the intersection of the event $\{X_1(t) > \xi +c \text{ for every $t$}\}$ and the event $\{\sup_{0\leq s \leq t} \, (-X_2(s))^{+} < c /r_2 \text{ for every $t$} \}$ for every positive $c$, both of which correspond to the first hitting problems of one-dimensional Brownian motion with a drift. We will use the idea repeatedly in the proofs of Theorem \ref{thm:uvtendtoinfty} and Lemma \ref{lemma:boundtau}.
\end{remark}
}

We now turn to Proposition \ref{proposition:hittingtime} below, \rems{which is a reformulation of the formula 1.2.4(1) on p. 252 of \cite{Borodin}.}

\begin{proposition} \label{proposition:hittingtime}
Let $B(t)$ be a one dimensional Brownian motion started from the origin under $\PP$. For $\mu >0$ and $x>0$, we have
\begin{equation*}
\PP(B(t) + \mu t > -x \text{ for every $t \geq 0$}) = 1- e^{-2x\mu}.
\end{equation*}
\end{proposition}

With Lemma \ref{lemma:crucial1} and Proposition \ref{proposition:hittingtime} in hand, we state Theorem \ref{thm:uvtendtoinfty} below.
\begin{theorem}
%We have
%\begin{equation*}
%\lim_{(u,v) \rightarrow \infty} \, \PP_{(u,v)} [T_{1/n} = \infty \text{ for some n $\in$ $\mathbb{N}_{+}$}  ] =1.
%\end{equation*}
%In particular 
When the starting point tends to infinity, the probability that \rems{the} process does not hit the origin tends to one. Namely,
\begin{equation*}
\lim_{\rems{\|(u,v)\|} \rightarrow \infty} \, \PP_{(u,v)} [T = \infty ] =1.
\end{equation*}
Equivalently,
$$
\lim_{\rems{\|(u,v)\|} \rightarrow \infty} \, \PP_{(u,v)} [T < \infty ] =0.
$$
\label{thm:uvtendtoinfty}
\end{theorem}

\begin{proof}

Fix $\xi >0$. For $\|(u,v)\|$ sufficiently large, we have $u >2 \xi$ or $v>2 \xi$. If $u >2 \xi$, by Lemma~\ref{lemma:crucial1}, we have
\begin{eqnarray*}
&&\PP_{(u,v)} [T = \infty \text{ \remst{for some n $\in$ $\mathbb{N}_{+}$}} ]
\geq \PP_{(u,v)} [\tau_{1}^{\xi} = \infty]  \\
&\geq& \PP_{(u,v)} \left[X_{1}(t) - r_2 \, \sup_{0 \leq s \leq t} (-X_{2}(s))^{+} > \xi \text{ for every $t \geq 0$}\right]  \\
&\geq& \PP_{(u,v)} [ X_{1}(t) > \xi + u/2 \text{ for every $t \geq 0$ and } X_{2}(t) > - u/(2r_2) \text{ for every $t \geq 0$}]  \\
&\geq& \PP_{(u,v)} [ X_{1}(t) > \xi + u/2 \text{ for every $t \geq 0$}] \\
&& + \PP_{(u,v)} [ X_{2}(t) > - u/(2r_2) \text{ for every $t \geq 0$}] -1  \\
&=& \PP_{(u,v)} [W_{1}(t) + \mu_1 t > -(u - 2\xi)/2 \text{ for every $t \geq 0$}]  \\
&& + \PP_{(u,v)} [ W_{2}(t) + \mu_2 t > - u/(2r_2) -v \text{ for every $t \geq 0$}] -1 \\
&=& 1- e^{-(u-2\xi)\mu_1} + 1- e^{-(u/r_2  + 2v) \mu_2} -1  \\
&=& 1- e^{-(u-2\xi)\mu_1} - e^{-(u/r_2  + 2v) \mu_2},
\end{eqnarray*}
where the second to last equality invokes Proposition \ref{proposition:hittingtime}. Similarly, if $v > 2 \xi$, we have
\begin{equation*}
\PP_{(u,v)} [T = \infty \text{ \remst{for some n $\in$ $\mathbb{N}_{+}$}}  ]
\geq 1- e^{-(v-2\xi)\mu_2} - e^{-(v/r_1  + 2u) \mu_1}.
\end{equation*}
Hence,
\rems{
\begin{eqnarray*}
&&
 \PP_{(u,v)} [T = \infty] 
 \\ & 
\geq& \max\{ (1- e^{-(u-2\xi)\mu_1} - e^{-(u/r_2  + 2v) \mu_2}) 1_{\{u > 2 \xi\}} , (1- e^{-(v-2\xi)\mu_2} - e^{-(v/r_1  + 2u) \mu_1}) 1_{\{v > 2 \xi\}}  \}.
\end{eqnarray*}}
Letting $(u,v)$ tend to $\infty$, the desired result follows.
\end{proof}

We now turn to Proposition \ref{contradiction} below, which shall be needed to prove Theorem \ref{thm:lim0}.
\begin{proposition}\label{contradiction}
We have the following subset relationship
\begin{eqnarray*}
\{ u + W_{1}(t) + \mu_1 t <0 \text{ and }  v + W_{2}(t) + \mu_2 t <0, \text{ for some $t \in \mathbb{R}_{+}$}\} \subset \{T < \infty\}.
\end{eqnarray*}
\end{proposition}

\begin{proof}
We prove this claim by contradiction. For the sake of contradiction, let us fix $ \omega \in  \{ u + W_{1}(t) + \mu_1 t <0 \text{ and }  v + W_{2}(t) + \mu_2 t <0, \text{ for some $t \in \mathbb{R}_{+}$}\}\cap  \{T = \infty\}$. Assuming $T= \infty$, the process can be written as
$$
\begin{cases}
Z_1(t)=u+W_1(t)+\mu_1 t+ l_1(t) -r_2 l_2(t),
\\
Z_2(t)=v+W_2(t)+\mu_2 t - r_1 l_1(t) +l_2(t).
\end{cases}
$$
Solving the linear system for $l_1$ and $l_2$, we obtain
$$
\begin{cases}
{(r_1 r_2 -1)}l_1(t)={(u+W_1(t)+\mu_1 t -Z_1(t)) +r_2 (v+W_2(t)+\mu_2 t -Z_2(t)) } ,
\\
{(r_1 r_2 -1)} l_2(t)={ r_1 (u+W_1(t)+\mu_1 t -Z_1(t)) + (v+W_2(t)+\mu_2 t -Z_2(t)) }.
\end{cases}
$$
For all $t \in \mathbb{R}_{+}$ such that
$$u+W_1(t)+\mu_1 t <0,$$ and
$$
v+W_2(t)+\mu_2 t <0,$$
we have
${(r_1 r_2 -1)}l_1(t) <0$ and ${(r_1 r_2 -1)}l_2(t)<0$, which is not possible since $l_1(t)$ and $l_2(t)\geqslant 0$ and as we assumed $(r_1 r_2 -1)\geqslant 0$. A contradiction has been reached. 
\end{proof}

Theorem \ref{thm:lim0} below considers the behavior of the process when the starting point tends to the origin.
\begin{theorem}\label{thm:lim0}
%Let $(u_n,v_n)\in\mathbb{R}_+^2$ a sequence of points such that $\| (u_n,v_n) \| \to 0$ when $n\to\infty$. 
When the starting point tends to the origin, the probability that the process hits the origin in finite time tends towards one. That is,
$$\lim_{(u,v) \rightarrow (0,0)} \, \PP_{(u,v)} \left[
T<\infty
\right] 
=
 1,
 $$ 
or equivalently,
 $$\lim_{(u,v) \rightarrow (0,0)} \, \PP_{(u,v)} \left[
T = \infty
\right] =
 0.
 $$ 
%$$\mathbb{P}  \left[\sup_{t \geqslant 0} |Z(t)|=\infty   \vert T=\infty \right] =1.$$
\end{theorem}

\begin{proof}

By Proposition \ref{contradiction}, we have that
$$
\PP_{(u,v)} \left[
T<\infty
\right] \geqslant 
\PP \left[\exists t \in \mathbb{R}_{+}:
u+W_1(t)+\mu_1 t <0
\text{ and }
v+W_2(t)+\mu_2 t <0
\right].
 $$
By the properties of planar Brownian motion, we have
 $$ 
\PP \left[\exists t \in \mathbb{R}_{+}:
W_1(t)+\mu_1 t <0
\text{ and }
W_2(t)+\mu_2 t <0
\right]=1.
 $$ 
 Let $(u_n,v_n)\in\mathbb{R}_+^2$ be a sequence of points such that $ (u_n,v_n) \rightarrow \rems{(0,0)}$.
Note that
\begin{equation*}
\begin{aligned}
&\bigcup_{n=1}^{\infty} \bigcap_{m=n}^{\infty} \left\{\exists t \in \mathbb{R}_{+}:
u_n+W_1(t)+\mu_1 t <0
\text{ and }
v_n+W_2(t)+\mu_2 t <0 \right\} \\
\supset& \left\{ \exists t \in \mathbb{R}_{+}:
W_1(t)+\mu_1 t <0
\text{ and }
W_2(t)+\mu_2 t <0 \right\}.
\end{aligned}
\end{equation*}
Applying Fatou's Lemma yields
$$ \liminf_{n \rightarrow \infty} \PP \left[\exists t \in \mathbb{R}_{+}:
u_n+W_1(t)+\mu_1 t <0
\text{ and }
v_n+W_2(t)+\mu_2 t <0
\right]
\geqslant
 1.
 $$ 
We may therefore conclude that
$$\PP_{(u_n,v_n)} \left[
T<\infty
\right] 
\underset{n \to \infty}{\longrightarrow}
 1,
 $$
and the desired result follows. 
\end{proof}

\subsection{\rems{Complementarity of absorption and escape}}
%Proving Theorem \ref{thm:Ztoinfty}
We now turn to Theorem \ref{thm:Ztoinfty}, which \rems{states} that the process has only two possible behaviors: either $T<\infty$, \remst{which means that the process is absorbed at the origin in finite time,} or $T=\infty$, in which case \remst{the process escapes to infinity, namely }$Z(t)\to\infty$ when $t\to\infty$. The result first requires the proofs of \rems{three auxiliary statements}
%Proposition \ref{proposition:exittime}, Lemma \ref{lemma:boundtau}, and Lemma \ref{lemma:infty1/n} }, 
which we give below.

\begin{proposition} \label{proposition:exittime}
Suppose $B(t)$ is a one dimensional Brownian motion starting from the origin under the measure $\PP$. Let $a,b$ be two positive numbers. Then
\begin{equation*}
\PP (-a -b t < B(t) < a + b t \text{ for every $t\geq 0$} ) >0.
\end{equation*}
\end{proposition}
\begin{proof}
Let $\lambda = \ln2 /(2b) +1$. Note that $1- 2\,e^{-2\lambda b} >0$. Then
\begin{eqnarray*}
&&\PP( -\lambda -b t < B(t) < \lambda + b t \text{ for every $t\geq 0$} )  \\
&\geq& \PP(  B(t)  >- \lambda -b t \text{ for every $t\geq 0$} ) 
+ \PP( B(t) < \lambda + b t \text{ for every $t\geq 0$} ) -1  \\
&=& 2 (1- e^{-2\lambda b}) -1 = 1- 2\,e^{-2\lambda b} >0.
\end{eqnarray*}
Let $H_a := \inf \{ t: |B(t)| =a \}$. By standard exit time properties of Brownian motion, $\PP(H_a > \lambda/b +1)> 0$. Then
\begin{eqnarray*}
&& \PP(-a-b t < B(t) < a + b t \text{ for every $t \geq 0$}) \\
&=& \PP(H_a > \lambda/b +1) \,\PP(-a-b t < B(t) < a + b t \text{ for every $t \geq 0$} \ | H_a > \lambda /b +1). 
\end{eqnarray*}
By the strong Markov property of Brownian motion,
\begin{eqnarray*}
&&\PP(-a-b t < B(t) < a + b t , \forall t \ | H_a > \lambda /b +1)  \\
&=& \PP(-a-b (t+H_a) < B(t + H_a) < a + b (t + H_a)  , \forall t  \ | H_a > \lambda /b +1)  \\
&=&  \PP(-a-b (t+H_a) -B(H_a) < B(t + H_a) -B(H_a) < a + b (t + H_a) -B(H_a)  , \forall t \  | H_a > \lambda /b +1)  \\
&\geq& \PP(-\lambda - b t < B(t + H_a) -B(H_a) < \lambda + b t  , \forall t \ | H_a > \lambda /b +1) \\
&=& \PP( -\lambda -b t < B(t) < \lambda + b t, \forall t \,) \\
&>& 0,
\end{eqnarray*}
from which the desired result follows.
\end{proof}

We now turn to Lemma \ref{lemma:boundtau}.

\begin{lemma} \label{lemma:boundtau}
For $\alpha$ a positive number, 
\begin{eqnarray}
&& \inf_{u \geq \alpha} \, \rems{\PP}_{(u,0)} [\tau_{1}^{0} = \infty] >0, \label{eq:inf>01} \\
&& \inf_{v \geq \alpha} \, \rems{\PP}_{(0,v)} [\tau_{2}^{0} = \infty] >0. \label{eq:inf>02}
\end{eqnarray}
\end{lemma}

\begin{proof}
We need only prove \eqref{eq:inf>01}, since the proof of \eqref{eq:inf>02} is completely symmetric. Let us consider $\xi < \alpha$. By Lemma \ref{lemma:crucial1},
\begin{eqnarray}
&&\PP_{(u,0)} [\tau_{1}^{0} = \infty]
\geq \PP_{(u,0)} [\tau_{1}^{\xi} = \infty] \notag \\
&\geq& \PP_{(u,0)} \left[X_{1}(t) - r_2 \, \sup_{0 \leq s \leq t} (-X_{2}(s))^{+} > \xi \text{ for every $t \geq 0$}\right]  \notag \\
&=& \PP_{(u,0)} \left[ u + W_{1}(t) + \mu_1 t - r_2 \, \sup_{0 \leq s \leq t} (-W_{2}(s) - \mu_2 t)^{+} > \xi \text{ for every $t \geq 0$}\right] \notag \\
&\geq& \PP_{(u,0)} [W_{1}(t) + \mu_1 t > -(u - \xi)/2 \text{ for every $t \geq 0$} \notag \\
&& \ \ \ \ \ \ \ \ \text{and } W_{2}(t) + \mu_2 t > -(u - \xi)/(2 r_2) \text{ for every $t \geq 0$}]. \label{eq:2.1}
\end{eqnarray}
Let $B_{1}(t)$ and $B_{2}(t)$ be two independent Brownian motions starting from $0$ under $\PP_{(u,0)}$. Then, under $\PP_{(u,0)}$, the process $(W_{1}(t), W_{2}(t))$ has the same law as $( B_{1}(t) , \rho B_{1}(t) + \sqrt{1 - \rho^2} B_{2}(t) )$. We now show that \eqref{eq:inf>01} holds in three separate cases: $\rho =0$, $0 < \rho <1 $ and $-1 <\rho <0$.\\

\noindent \underline{Case I: $\rho=0$.}
If $\rho =0$, then $W_{1}(t)$ and $W_{2}(t)$ are two independent Brownian motions. Then
\begin{eqnarray*}
\eqref{eq:2.1}& =& \PP_{(u,0)} [W_{1}(t) + \mu_1 t > -(u - \xi)/2 \text{ for every $t \geq 0$}] \\
&& \times \, \PP_{(u,0)} [ W_{2}(t) + \mu_2 t > -(u - \xi)/(2 r_2) \text{ for every $t \geq 0$}]  \\
&=& \left( 1 - e^{-(u - \xi) \mu_1} \right) \cdot \left( 1- e^{-(u-\xi)\mu_2/r_2} \right),
\end{eqnarray*}
where the last equality invokes Proposition \ref{proposition:hittingtime}. Taking infimums yields
\begin{equation*}
\inf_{u \geq \alpha} \, \rems{\PP}_{(u,0)} [\tau_{1}^{0} = \infty]
\geq \left( 1 - e^{-(\alpha - \xi) \mu_1} \right) \cdot \left( 1- e^{-(\alpha-\xi)\mu_2/r_2} \right)
>0.
\end{equation*}

\noindent \underline{Case II: $0 < \rho <1$.}
If $0 < \rho <1$, then
\begin{eqnarray*}
\eqref{eq:2.1}& =& \PP_{(u,0)} [B_{1}(t) + \mu_1 t > -(u - \xi)/2 \text{ for every $t \geq 0$}  \\
&& \ \ \ \ \ \ \ \ \text{and } \rho B_{1}(t) + \sqrt{1 - \rho^2} B_{2}(t) + \mu_2 t > -(u - \xi)/(2 r_2) \text{ for every $t \geq 0$}]  \\
&\geq& \PP_{(u,0)} [B_{1}(t) + (\mu_1 \wedge \mu_2) t > -(u - \xi)/(2 r_2) \text{ for every $t \geq 0$} \notag \\
&& \ \ \ \ \ \ \ \ \text{and } \sqrt{1 - \rho^2} B_{2}(t) + (1-\rho)\mu_2 t > -(1-\rho)(u - \xi)/(2 r_2) \text{ for every $t \geq 0$}].
\end{eqnarray*}
Using the same argument in the case for $\rho = 0$, \eqref{eq:inf>01} follows.\\

\noindent \underline{Case III: $-1 < \rho <0$.}
If $-1 < \rho <0$, then for $u \geq \alpha$
\begin{eqnarray*}
\eqref{eq:2.1}& =& \PP_{(u,0)} [B_{1}(t) + \mu_1 t > -(u - \xi)/2 \text{ for every $t \geq 0$}  \\
&& \ \ \ \ \ \ \ \ \text{and } \rho B_{1}(t) + \sqrt{1 - \rho^2} B_{2}(t) + \mu_2 t > -(u - \xi)/(2 r_2) \text{ for every $t \geq 0$}]  \\
& \geq& \PP_{(u,0)} [B_{1}(t) + \mu_1 t > -(u - \xi)/2 \text{ for every $t \geq 0$} , \\
&& \ \ \ \ \ \ \ \  \rho B_{1}(t) - \rho (\mu_1 \wedge \mu_2) t > - |\rho|(u - \xi)/(2 r_2) \text{ for every $t \geq 0$}  \\
&&\text{and } \sqrt{1 - \rho^2} B_{2}(t) + (\mu_2+ \rho(\mu_1 \wedge \mu_2) t > -(1-|\rho|)(u - \xi)/(2 r_2) \text{ for every $t \geq 0$}] \\
& \geq& \PP_{(u,0)} [-(u - \xi)/(2r_2) - (\mu_1 \wedge \mu_2) t <B_{1}(t) < (u - \xi)/(2r_2) +(\mu_1 \wedge \mu_2) t, \forall t  \\
&&\text{and } \sqrt{1 - \rho^2} B_{2}(t) + (\mu_2+ \rho(\mu_1 \wedge \mu_2) t > -(1-|\rho|)(u - \xi)/(2 r_2), \forall t] \\ 
&=& \PP_{(u,0)} [-(u - \xi)/(2r_2) - (\mu_1 \wedge \mu_2) t <B_{1}(t) < (u - \xi)/(2r_2) +(\mu_1 \wedge \mu_2) t, \forall t]  \\
&& \times \, \PP_{(u,0)}[\sqrt{1 - \rho^2} B_{2}(t) + (\mu_2+ \rho(\mu_1 \wedge \mu_2) t > -(1-|\rho|)(u - \xi)/(2 r_2), \forall t]  \\
&\geq& \PP_{(u,0)} [-(\alpha - \xi)/(2r_2) - (\mu_1 \wedge \mu_2) t <B_{1}(t) < (\alpha - \xi)/(2r_2) +(\mu_1 \wedge \mu_2) t, \forall t]  \\
&& \times \, \PP_{(u,0)}[\sqrt{1 - \rho^2} B_{2}(t) + (\mu_2+ \rho(\mu_1 \wedge \mu_2) t > -(1-|\rho|)(\alpha - \xi)/(2 r_2), \forall t].
\end{eqnarray*}
Taking infimums and invoking Proposition \ref{proposition:exittime}, \eqref{eq:inf>01} follows. This concludes the proof.
\end{proof}

Let us denote $T_{r} := \inf\{t\geq 0 : \|Z(t \wedge T)\| \leq r\} $.

\begin{lemma} \label{lemma:infty1/n}
For fixed $n$, on the event $\{T_{1/n} = \infty\}$, we have $\PP_{(u,v)}$-a.s. 
\begin{equation*}
\lim_{t \rightarrow \infty} Z(t) = \infty.
\end{equation*}
That is,
\begin{equation}
\PP_{(u,v)} \left[\liminf_{t \rightarrow \infty} \, Z(t) < \infty, T_{\frac{1}{n}} = \infty\right] = 0. \label{eq:tendinfty}
\end{equation}
\end{lemma}

\begin{proof}
We will first show \eqref{eq:tendinfty} holds when $v=0$. Then \eqref{eq:tendinfty} will follow immediately in the case\remst{ that} $u=0$. We conclude by showing that \eqref{eq:tendinfty} holds when $u \neq 0$ and $v \neq 0$.\\

\noindent \underline{Case I: $v=0$.}
When $v=0$, let
\begin{equation*}
K := \sup_{u \geq 0} \, \PP_{(u,0)} \left[\liminf_{t \rightarrow \infty} \, Z(t) < \infty, T_{\frac{1}{n}} = \infty\right].
\end{equation*}
For $u \leq 1/n$,
\begin{equation*}
\PP_{(u,0)} \left[\liminf_{t \rightarrow \infty} \, Z(t) < \infty, T_{\frac{1}{n}} = \infty\right] = 0.
\end{equation*}
Then
\begin{equation}
K = \sup_{u \geq 1/n} \, \PP_{(u,0)} \left[\liminf_{t \rightarrow \infty} \, Z(t) < \infty, T_{\frac{1}{n}} = \infty\right].  \label{eq:3.1}
\end{equation}
We now define a stopping time
\begin{equation*}
\eta_{1}^{0} := \begin{cases}
 \inf \{ t \geq \tau_{1}^{0} :  Z_{2}(t) =0 \}, &  \tau_{1}^{0} < \infty,  \\
\infty , & \tau_{1}^{0} = \infty .
\end{cases}
\end{equation*}
By Lemma \ref{lemma:boundtau},
\begin{equation*}
\inf_{u \geq 1/n} \,\PP_{(u,0)} [\eta_{1}^{0} = \infty]
\geq \inf_{u \geq 1/n} \,\PP_{(u,0)} [\tau_{1}^{0} = \infty]
>0,
\end{equation*}
and hence,
\begin{equation}
\sup_{u \geq 1/n} \PP_{(u,0)}[\eta_{1}^{0} < \infty] <1. \label{eq:3.2}
\end{equation}
Note that
\begin{eqnarray}
&&
\PP_{(u,0)} \left[\liminf_{t \rightarrow \infty} \, Z(t) < \infty, T_{\frac{1}{n}} = \infty\right]  \notag \\
&=& \PP_{(u,0)} \left[\tau_{1}^{0} = \infty, \liminf_{t \rightarrow \infty} \, Z(t) < \infty, T_{\frac{1}{n}} = \infty\right] \notag \\
&& +\  \PP_{(u,0)} \left[\tau_{1}^{0} < \infty, \eta_{1}^{0} = \infty, \liminf_{t \rightarrow \infty} \, Z(t) < \infty, T_{\frac{1}{n}} = \infty\right] \notag \\
&& + \ \PP_{(u,0)} \left[\eta_{1}^{0} < \infty, \liminf_{t \rightarrow \infty} \, Z(t) < \infty, T_{\frac{1}{n}} = \infty\right]. \label{eq:decompostion3}
\end{eqnarray}
On the event $\{\tau_{1}^{0} = \infty\}$,  for all $t \geq 0$, $T= \infty$ and $l_1(t) = 0$. Then
\begin{eqnarray*}
Z_{2}(t) = X_{2}(t) + l_{2}(t) \geq X_{2}(t) 
=W_{2}(t) + \mu_2 t
\rightarrow \infty,
\end{eqnarray*}
$\PP_{(u,0)}$-a.s., by the law of the iterated logarithm for Brownian motion. Hence, the first term on the right-hand side of \eqref{eq:decompostion3} is $0$. We now consider the second term on the right-hand side of \eqref{eq:decompostion3}. On the event $\{ \tau_{1}^{0} < \infty \}$, let us define $\tilde{\eta}_{1}^{0} : = \inf \{ t\geq 0 : Z_{2}(t + \tau_{1}^{0}) =0\}$ and $ \tilde{T}_{1/n} := \inf \{t\geq 0 : \|Z(t + \tau_{1}^{0})\| \leq 1/n \}$. By the strong Markov property, we have
\begin{eqnarray*}
&& \PP_{(u,0)} \left[\tau_{1}^{0} < \infty, \eta_{1}^{0} = \infty, \liminf_{t \rightarrow \infty} \, Z(t) < \infty, T_{\frac{1}{n}} = \infty\right] \\
&=& \PP_{(u,0)} \left[\tau_{1}^{0} < \infty, \inf_{0 \leq s \leq \tau_{1}^{0}} \|Z(s)\|> \frac{1}{n}, \tilde{\eta}_{1}^{0} = \infty, \liminf_{t \rightarrow \infty} \, Z(t+ \tau_{1}^{0}) < \infty, \tilde{T}_{\frac{1}{n}} = \infty\right]   \\
&=& \EE_{(u,0)} \left[ \mathbbm{1}_{\{\tau_{1}^{0} < \infty, \inf_{0 \leq s \leq \tau_{1}^{0}} \|Z(s)\|> 1/n\}} 
\, \PP_{Z(\tau_{1}^{0})} \left[\eta_{1}^{0} = \infty, \liminf_{t \rightarrow \infty} \, Z(t) < \infty, T_{\frac{1}{n}} = \infty \right] \right] \\
&=& 0.
\end{eqnarray*}
By the same argument used to show that the first term on the right-hand side of \eqref{eq:decompostion3} is $0$, for $v >0$,
\begin{equation*}
\PP_{(0,v)} \left[\eta_{1}^{0} = \infty, \liminf_{t \rightarrow \infty} \, Z(t) < \infty, T_{\frac{1}{n}} = \infty \right] =0.
\end{equation*}
This proves that the second term on the right-hand side of \eqref{eq:decompostion3} is also $0$. We now consider the third term on the right-hand side of \eqref{eq:decompostion3}. On the event $\{ \eta_{1}^{0} < \infty \}$, let $\hat{T}_{1/n} := \inf \{ t \geq 0 : Z(t + \eta_{1}^{0}) \leq 1/n \}$. By the strong Markov property,
\begin{eqnarray*}
&&\PP_{(u,0)} \left[\eta_{1}^{0} < \infty, \liminf_{t \rightarrow \infty} \, Z(t) < \infty, T_{\frac{1}{n}} = \infty\right] \\
&=& \PP_{(u,0)} \left[\eta_{1}^{0} < \infty, \inf_{0 \leq s \leq \eta_{1}^{0}} \|Z(s)\|> \frac{1}{n},  \liminf_{t \rightarrow \infty} \, Z(t +\eta_{1}^{0}) < \infty, \hat{T}_{\frac{1}{n}} = \infty\right]  \\
&=& \EE_{(u,0)} \left[ \mathbbm{1}_{\{\eta_{1}^{0} < \infty, \inf_{0 \leq s \leq \eta_{1}^{0}} \|Z(s)\|> 1/n\}} 
\, \PP_{Z(\eta_{1}^{0})} \left[\liminf_{t \rightarrow \infty} \, Z(t) < \infty, T_{\frac{1}{n}} = \infty \right] \right]  \\
&\leq& K \cdot \EE_{(u,0)} \left[ \mathbbm{1}_{\{\eta_{1}^{0} < \infty, \inf_{0 \leq s \leq \eta_{1}^{0}} \|Z(s)\|> 1/n\}} \right]  \\
&\leq& K \cdot \PP_{(u,0)} [\eta_{1}^{0} < \infty].
\end{eqnarray*}
Combining \eqref{eq:decompostion3} and the above estimates yields
\begin{equation*}
\PP_{(u,0)} \left[\liminf_{t \rightarrow \infty} \, Z(t) < \infty, T_{\frac{1}{n}} = \infty\right]
\leq K \cdot \PP_{(u,0)} [\eta_{1}^{0} < \infty].
\end{equation*}
Taking supremums and invoking \eqref{eq:3.1}, we obtain
\begin{eqnarray*}
K = \sup_{u \geq 1/n} \, \PP_{(u,0)} \left[\liminf_{t \rightarrow \infty} \, Z(t) < \infty, T_{\frac{1}{n}} = \infty\right]
\leq K \cdot \sup_{u \geq 1/n} \PP_{(u,0)} [\eta_{1}^{0} < \infty].
\end{eqnarray*}
Together with \eqref{eq:3.2}, we have $K=0$. Hence, for every $u \geq 0$,
\begin{equation}
\PP_{(u,0)} \left[\liminf_{t \rightarrow \infty} \, Z(t) < \infty, T_{\frac{1}{n}} = \infty\right] = 0. \label{eq:3.3}
\end{equation}
Similarly, for every $v \geq 0$,
\begin{equation}
\PP_{(0,v)} \left[\liminf_{t \rightarrow \infty} \, Z(t) < \infty, T_{\frac{1}{n}} = \infty\right] = 0. \label{eq:3.4}
\end{equation}

\noindent \underline{Case II: $u \neq 0$ and $v \neq 0$.} For the case when $u \neq 0$ and $v \neq 0$, let $\tau := \inf \{t \geq 0 : Z_{1}(t) =0 \text{ or } Z_{2}(t) =0\}$. Then
\begin{eqnarray}
&& \PP_{(u,v)} \left[\liminf_{t \rightarrow \infty} \, Z(t) < \infty, T_{\frac{1}{n}} = \infty\right] \notag \\
&=& \PP_{(u,v)} \left[ \tau = \infty, \liminf_{t \rightarrow \infty} \, Z(t) < \infty, T_{\frac{1}{n}} = \infty\right] \notag \\
&& + \  \PP_{(u,v)} \left[ \tau < \infty, \liminf_{t \rightarrow \infty} \, Z(t) < \infty, T_{\frac{1}{n}} = \infty\right]. \label{eq:decompostion2}
\end{eqnarray}
On the event $\{\tau = \infty\}$, $T = \infty$ and, for every $t \geq 0$, $l_1(t)=l_2(t)=0$. Then, as $t \rightarrow \infty$,
\begin{equation*}
Z_{1}(t) = u + W_{1}(t) + \mu_1 t \rightarrow \infty,
\end{equation*}
 $\PP_{(u,v)}$-a.s. Hence the first term on the right-hand side of \eqref{eq:decompostion2} is $0$. We now consider the second term  on the right-hand side of \eqref{eq:decompostion2}. By the strong Markov property,
\begin{eqnarray*}
&& \PP_{(u,v)} \left[ \tau < \infty, \liminf_{t \rightarrow \infty} \, Z(t) < \infty, T_{\frac{1}{n}} = \infty\right]  \\
&\leq& \EE_{(u,v)} \left[ \mathbbm{1}_{\{\tau < \infty\}} \,\PP_{Z(\tau)} \left[\liminf_{t \rightarrow \infty} \, Z(t) < \infty, T_{\frac{1}{n}} = \infty\right]  \right]  \\
&=& 0,
\end{eqnarray*}
where \eqref{eq:3.3} and \eqref{eq:3.4} have been invoked in the last equality. Hence the second term on the right-hand side of \eqref{eq:decompostion2} is also $0$. Thus for $u \neq 0$ and $v \neq 0$,
\begin{equation*}
\PP_{(u,v)} \left[\liminf_{t \rightarrow \infty} \, Z(t) < \infty, T_{\frac{1}{n}} = \infty\right] = 0.
\end{equation*}
The proof is now complete.
\end{proof}

%\begin{lemma} \label{prop:inftyT}
%On the event $\{T_{1/n} = \infty \text{ for some $n \in \mathbb{N}_{+}$}\}$, we have $\PP_{(u,v)}$-a.s.
%\begin{equation*}
%\lim_{t \rightarrow \infty} Z(t) = \infty,
%\end{equation*}
%namely,
%\begin{equation*}
%\PP_{(u,v)} \left[\liminf_{t \rightarrow \infty} \, Z(t) < \infty, T_{\frac{1}{n}} = \infty \text{ for some $n \in \mathbb{N}_{+}$}\right] = 0. 
%\end{equation*}
%\end{lemma}
%\begin{proof}
%This follows directly from Proposition \ref{lemma:infty1/n}.
%\end{proof}

With the above results in hand, we are now ready to state Theorem \ref{thm:Ztoinfty}.

\begin{theorem}\label{thm:Ztoinfty}
On the event $\{T = \infty \}$, $\PP_{(u,v)}$-a.s. the process $Z(t)$ tends to infinity when $t\to\infty$, 
namely
$$
\PP_{(u,v)} 
\left. \left[
\lim_{t \rightarrow \infty} \, Z(t) = \infty
\right\vert 
T=\infty
\right]
=1.
$$
Equivalently,
$$
\PP_{(u,v)} 
 \left[
\liminf_{t \rightarrow \infty} \, Z(t) < \infty
, \,
T=\infty
\right]
=0.
$$

\end{theorem}
\begin{proof}
We deduce from Lemma~\ref{lemma:infty1/n} that for every $n \in \mathbb{N}_{+}^*$
\begin{eqnarray*}
&&\PP_{(u,v)} 
 \left[
\liminf_{t \rightarrow \infty} \, Z(t) < \infty
, \,
T=\infty
\right] \\
&=& \PP_{(u,v)} 
 \left[
\liminf_{t \rightarrow \infty} \, Z(t) < \infty
, \, T_{\frac{1}{n}} < \infty, \,
T=\infty
\right].
\end{eqnarray*}
Applying the strong Markov property yields
\begin{eqnarray*}
&& \PP_{(u,v)} 
 \left[
\liminf_{t \rightarrow \infty} \, Z(t) < \infty
, \, T_{\frac{1}{n}} < \infty, \,
T=\infty
\right]  \\
&=& \mathbb{E}_{(u,v)} \left[ \mathbbm{1}_{\{T_{1/n} < \infty\}} \ 
\PP_{Z(T_{1/n})} \left[ \liminf_{t \rightarrow \infty} \, Z(t) < \infty , \, T =\infty \right] \right] \\
&\leq& \sup_{\|(u,v)\| = 1/n} \, \PP_{(u,v)} \left[ \liminf_{t \rightarrow \infty} \, Z(t) < \infty , \, T =\infty \right] \\
&\leq& \sup_{\|(u,v)\| = 1/n} \, \PP_{(u,v)} \left[  T =\infty \right].
\end{eqnarray*}
Applying Theorem~\ref{thm:lim0} and letting $n \rightarrow \infty$, the desired result follows.
\end{proof}

\section{Partial differential equation and functional equation}
\label{sec:functionalequation}

We now turn to the study of the escape probability $\mathbb{P}_{(u,v)}[T=\infty]$. We begin with Proposition~\ref{prop:PDE}, which provides partial differential equations characterizing the escape probability.  We proceed with Proposition~\ref{prop:functionaleq}, which gives a functional equation satisfied by the Laplace transform of the escape probability. \rems{Note that there is no particular difficulty in defining the process starting from the edge (except the origin)}.

Let us define the infinitesimal generator of the process inside the quarter plane as
$$
\mathcal{G} f (u,v) :=  \lim_{t\to 0} \frac{1}{t} (\mathbb{E}_{(u,v)} [f(Z(t \wedge T)] -f(u,v)),
$$
where $f$ must be a bounded function in the quadrant to ensure that the above limit exists and is uniform.
For $f$ twice differentiable, the infinitesimal generator inside the quadrant is
$$
\mathcal{G} f =
 \frac{1}{2} \left( \frac{\partial^2f }{\partial \rems{z_1^2}} + \frac{\partial^2f }{\partial \rems{z_2^2}} +2\rho \frac{\partial^2f }{\partial z_1 \partial z_2}  \right) +\mu_1 \frac{\partial f }{\partial z_1 } + \mu_2 \frac{\partial f }{\partial z_2 }  .
$$
\indent This leads us to Proposition \ref{prop:PDE}.
\begin{proposition}[Partial differential equation]
The absorption probability $$
 f(u,v)=\mathbb{P}_{(u,v)} [ T < \infty],
$$
is the only function which is both (i) bounded and continuous in the quarter plane and on its boundary and (ii) continuously differentiable in the quarter plane and on its boundary (except perhaps at the corner), and which satisfies the partial differential equation
$$
\mathcal{G} f (u,v)=0, \quad \forall (u,v)\in\mathbb{R}_+^2,
$$
with oblique Neumann boundary conditions
\begin{equation}
\begin{cases}
\partial_{r_1}f(0,v):=\frac{\partial f }{\partial u } (0,v)- r_1   \frac{\partial f }{\partial v } (0,v)=0 & \forall v>0,
\\
\partial_{r_2}f(u,0):=-r_2 \frac{\partial f }{\partial u} (u,0)+    \frac{\partial f }{\partial v } (u,0)=0 & \forall u>0,
\end{cases}
\label{eq:Neumanboundcond}
\end{equation}
and with limit values
$$
\begin{cases}
f(0,0)=1, & 
\\
\lim_{(u,v)\to \infty} f(u,v) = 0.
\end{cases}
$$
The same result holds for the escape probability
$$g(u,v)= 1-f(u,v)=\mathbb{P}_{(u,v)} [ T =\infty]$$
but with the following limit values 
$$
\begin{cases}
g(0,0)=0, & 
\\
\lim_{\rems{\|(u,v)\|}\to \infty} f\rems{(u,v)} = 1.
\end{cases}
$$
\label{prop:PDE}
\end{proposition}
\rems{Remark that a similar partial differential equation with different limit values could be obtained for the domination probability considered in \cite{fomichov2020probability}.}
\begin{proof}
This proof is inspired by \citet[p. 86-89]{foddy_analysis_1984}.
We assume that $f$  satisfies the hypotheses of the Proposition.
%be bounded, continuous in the quarter plane and its boundary, continuously differentiable in the quarter plane and its boundary (except maybe in the corner) which satisfy 
%$$
%\begin{cases}
%\mathcal{G} f (z_1,z_2)=0 & \forall z_1>0, \ \forall z_2>0,
%\\
%\partial_{r_1}f(0,z_2):=\frac{\partial f }{\partial z_1 } (0,z_2)- r_1   \frac{\partial f }{\partial z_2 } (0,z_2)=0 & \forall z_2>0,
%\\
%\partial_{r_2}f(z_1,0):=-r_2 \frac{\partial f }{\partial z_1 } (z_1,0)+    \frac{\partial f }{\partial z_2 } (z_1,0)=0 & \forall z_1>0,
%\\
%f(0,0)=1, & 
%\\
%\lim_{|z|\to \infty} f(z) = 0.
%\end{cases}
%$$
Applying Dynkin's formula, we obtain
\begin{align*}
\mathbb{E}_{(u,v)} [f(Z(t \wedge T))]
&= f(u,v)+ \mathbb{E}_{(u,v)} \int_0^{t \wedge T} \mathcal{G} f(Z(s))\, \mathrm{d}s  
+ \sum_{i=1}^2 \mathbb{E}_{(u,v)} \int_0^{t \wedge T}
\partial_{r_i}f(Z(s)) \de l_i(s)
\\
&=f(u,v).
\end{align*}
But,
\begin{align*}
\mathbb{E}_{(u,v)} [f(Z(t \wedge T)]
 &= \mathbb{E}_{(u,v)} [f(Z(t \wedge T))  \mathds{1}_{T\leqslant t}]
+ \mathbb{E}_{(u,v)} [f(Z(t \wedge T)) \mathds{1}_{T > t}]
\\ &=  f(0,0) \mathbb{P}_{(u,v)} [ T \leqslant t]
+ \mathbb{E}_{(u,v)} [f(Z(t )) \mathds{1}_{T > t}]
\\ & \underset{t\to\infty}{\longrightarrow} \mathbb{P}_{(u,v)} [ T < \infty] 
+\lim_{t\to\infty} \mathbb{E}_{(u,v)} [f(Z(t )) \mathds{1}_{T > t}] 
\\ &= \mathbb{P}_{(u,v)} [ T < \infty] 
.
\end{align*}
Note that  $\underset{|z|\to \infty}{\lim} f(z) = 0$ and that for $T>t$, $Z(t)\underset{t\to\infty }{\rightarrow} \infty $ a.s. By dominated convergence and by Theorem ~\ref{thm:Ztoinfty},
$$\underset{t\to\infty}{\lim} \mathbb{E}_{(u,v)} [f(Z(t )) \mathds{1}_{T > t}] =\mathbb{E}_{(u,v)} [\underset{t\to\infty}{\lim} f(Z(t )) \mathds{1}_{T =\infty}] =0. $$
We may thus conclude that
$$
 f(u,v)=\mathbb{P}_{(u,v)} [ T < \infty].
$$
Conversely, denote $f(u,v):=\mathbb{P}_{(u,v)} [ T < \infty]$. The function $f$ is bounded. By the Markov property, we have
$$
\mathbb{E}_{(u,v)} [f(Z(t \wedge T)\rems{)}]
=f(u,v).
$$
Since $$\mathcal{G}f(\rems{u,v})= \lim_{t\to 0} \frac{1}{t} (\mathbb{E}_{(u,v)} [f(Z(t \wedge T)\rems{)}] -f(u,v))=0,$$ we may conclude that $\mathcal{G}f =0$ on the quarter plane.
The continuity and differentiability properties of $f$ are immediate from \cite[Thm 2.2 and Cor 2.4]{andres_pathwise_2009}. One can also refer to \cite{lipshutz2019} which establishes these properties in a greater generality. The Neumann boundary condition is satisfied by applying \cite[Cor 3.3]{andres_pathwise_2009}. The desired limit values at $0$ and at infinity are obtained by invoking Theorem~\ref{thm:uvtendtoinfty} and Theorem~\ref{thm:lim0}. The result for $g=1-f$ is straightforward, and this completes the proof.
\end{proof}

In preparation for Proposition \ref{prop:functionaleq}, let us define the Laplace transform of the escape probability starting from $(u,v)$ as
$$
\psi(x,y):=\int_0^\infty \int_0^\infty e^{-x u- y v} \mathbb{P}_{(u,v)}[T=\infty] \,\mathrm{d}u \mathrm{d}v.
$$
Further, let
\begin{equation}
\label{eq:defpsi12}
\psi_1(x):=\int_0^\infty e^{-x u} \mathbb{P}_{(u,0)}[T=\infty] \, \mathrm{d}u 
\quad
\text{and}
\quad
\psi_2(y):=\int_0^\infty e^{-y v}\mathbb{P}_{(0,v)}[T=\infty]  \, \mathrm{d}v.
\end{equation}
%These Laplace transforms are called moment generating functions.
We also define the kernel
\begin{equation}
K(x,y):=\frac{1}{2}(x^2+y^2+2\rho xy)+\mu_1 x +\mu_2 y\rems{,}
\label{eq:kernel}
\end{equation}
and let
\begin{equation}
k_1(x,y):=\frac{1}{2}(r_2 x +y)+\rho x+\mu_2
,
\quad
k_2(x,y):=\frac{1}{2}(x+r_1 y)+\rho y+\mu_1.
\label{eq:k1k2}
\end{equation}
\indent We now give a functional equation satisfied by the Laplace transform of the escape probability.
\begin{proposition}[Functional equation]
For $(x,y)\in\mathbb{C}^2$ such that $\Re x >0$ and $\Re y >0$ we have
\begin{equation}
K(x,y)\psi(x,y)
=
k_1(x,y) \psi_1(x)
 +
k_2(x,y) \psi_2(y).
\label{eq:functionalequation}
\end{equation}
\label{prop:functionaleq}
\end{proposition}
This functional equation recall\rems{s} the one obtained in~\cite[(32)]{fomichov2020probability} to compute an escape probability along one of the ax\rems{e}s for another range of parameters.
\begin{proof}
Recall the partial differential equation  in Proposition~\ref{prop:PDE} with the oblique Neumann boundary condition and limit values satisfied by $g(u,v):= \mathbb{P}_{(u,v)}[T=\infty]$.
Employing integration by parts yields
\begin{align*}
0&=\int_0^\infty \int_0^\infty e^{-xz_1-yz_2} \mathcal{G} g (z_1,z_2) \de z_1 \dd z_2
\\
0& = 
\int_0^\infty \frac{1}{2} e^{-y z_2} \left( 
-\frac{\partial g}{\partial z_1} (0,z_2) +x \int_0^\infty  
e^{-x z_1} \frac{\partial g}{\partial z_1} (z_1,z_2)
 \de z_1
\right) \de z_2
\\ &+
\int_0^\infty \frac{1}{2} e^{-x z_1} \left( 
-\frac{\partial g}{\partial z_2} (z_1,0) +y \int_0^\infty  
e^{-y z_2} \frac{\partial g}{\partial z_2} (z_1,z_2)
 \de z_2
\right) \de z_1
\\ &+
\int_0^\infty \rho e^{-x z_1} \left( 
-\frac{\partial g}{\partial z_1} (z_1,0) +y \int_0^\infty  
e^{-y z_2} \frac{\partial g}{\partial z_1} (z_1,z_2)
 \de z_2
\right) \de z_1
\\ &+
\int_0^\infty \mu_1 e^{-y z_2}
\left( 
- g (0,z_2) +x \int_0^\infty  
e^{-x z_1}  g (z_1,z_2)
 \de z_1
\right) \de z_2
\\ &+
\int_0^\infty \mu_2 e^{-x z_1}
\left( 
- g (z_1,0) +y \int_0^\infty  
e^{-y z_2}  g (z_1,z_2)
 \de z_2
\right) \de z_1
\\
0& = 
 - \frac{1}{2} r_1 \int_0^\infty  e^{-y z_2} 
\frac{\partial g}{\partial z_2}  (0,z_2) \de z_2 +
\frac{x}{2} \int_0^\infty  e^{-y z_2} \left(  -g(0,z_2) +
x \int_0^\infty  
e^{-x z_1} g (z_1,z_2)
 \de z_1
\right) \de z_2
\\ &- \frac{1}{2} r_2 \int_0^\infty  e^{-x z_1} 
\frac{\partial g}{\partial z_1}  (z_1,0) \de z_1 +
\frac{y}{2} \int_0^\infty  e^{-x z_1} \left(  -g(z_1,0) +
y \int_0^\infty  
e^{-y z_2} g (z_1,z_2)
 \de z_2
\right) \de z_1
\\ &-
\rho\int_0^\infty  e^{-x z_1}  
\frac{\partial g}{\partial z_1} (z_1,0) \de z_1 +
\rho y\int_0^\infty  e^{-y z_2} \left(
-g(0,z_2) + x\int_0^\infty  
e^{-x z_1} g (z_1,z_2)
 \de z_1
\right) \de z_2
\\ &- \mu_1
\int_0^\infty  e^{-y z_2}
 g (0,z_2) \de z_2 + \mu_1 x 
 \int_0^\infty \int_0^\infty  
e^{-x z_1 -y z_2}  g (z_1,z_2)
 \de z_1
 \dd z_2
\\ &- \mu_2
\int_0^\infty  e^{-x z_1}
 g (z_1,0) \de z_1 + \mu_2 y 
 \int_0^\infty \int_0^\infty  
e^{-x z_1 -y z_2}  g (z_1,z_2)
 \de z_1
 \dd z_2
\\
0& =
\left(\frac{1}{2}(x^2+y^2+2\rho xy)+\mu_1 x +\mu_2 y \right) \int_0^\infty \int_0^\infty e^{-xz_1-yz_2}  g (z_1,z_2) \de z_1 \dd z_2
\\ &-
\left(\frac{1}{2}(r_2 x +y)+\rho x+\mu_2 \right)
 \int_0^\infty e^{-xz_1}  g (z_1,0) \de z_1 
\\ &-
\left(\frac{1}{2}(x+r_1 y)+\rho y+\mu_1\right)
\int_0^\infty e^{-yz_2}  g (0,z_2) \de z_2 
\\ 
0&= K(x,y)\psi(x,y)
-
k_1(x,y) \psi_1(x)
-
k_2(x,y) \psi_2(y).
\end{align*}
This concludes the proof.
\end{proof}

\section{Kernel and asymptotics}
\label{sec:kernelasympt}

%\subsection{Kernel study}

We begin by studying some properties of the kernel $K$ defined in \eqref{eq:kernel}. Note that this kernel is similar to that in \cite{franceschi_2019} except that in the present paper the drift is positive.
We define the functions $X$ and $Y$ satisfying
$$
K(X(y),y)=0 \quad \text{and}\quad K(x,Y(x))=0.
$$
The branches are given by
\begin{equation}
\begin{cases}
X^\pm (y):=-(\rho y+\mu_1)\pm \sqrt{y^2(\rho^2-1)+ 2y(\mu_1\rho-\mu_2)+\mu_1^2},
\\
Y^\pm (x):=-(\rho x+\mu_2)\pm \sqrt{x^2(\rho^2-1)+ 2x(\mu_2\rho-\mu_1)+\mu_2^2}\rems{,}
\end{cases}
\label{eq:defXYpm}
\end{equation}
and the branch points of $X$ and $Y$ (which are roots of the
polynomials in the square roots of \eqref{eq:defXYpm}) are given, respectively, by
\begin{equation}
\begin{cases}
y^\pm:=\dfrac{\mu_1\rho-\mu_2\pm \sqrt{(\mu_1\rho-\mu_2)^2+\mu_1^2(1-\rho^2)}} {(1-\rho^2)},\\
x^\pm:=\dfrac{\mu_2\rho-\mu_1\pm \sqrt{(\mu_2\rho-\mu_1)^2+\mu_2^2(1-\rho^2)}} {(1-\rho^2)} .
\end{cases}
\label{eq:defxypm}
\end{equation}
By \eqref{eq:defangles} we obtain that
\begin{equation}
y^+=\mu_1 \frac{1-\cos(\beta-\theta)}{\sin\beta \sin(\beta-\theta)}.
\label{eq:y+angle}
\end{equation}
The functions $X^\pm$ and $Y^\pm$ are analytic, respectively, on the cut planes $\mbC\setminus ((-\infty,y^-]\cup [y^+,\infty))$ and $\mbC\setminus ((-\infty,x^-]\cup [x^+,\infty))$. Figure~\ref{fig:kernelBrownian} below depicts the functions $Y^\pm$ on $[x^-,x^+]$.

Recall $k_1$ and $k_2$ as defined in~\eqref{eq:k1k2}. Consider the intersection points between the ellipse $K=0$ and the lines $k_1=0$ and $k_2=0$. We define
\begin{equation}
\label{eq:x0y0}
x_0:= -{2\mu_1}<0
\quad \text{and}\quad
y_0:= -{2\mu_2}<0,
\end{equation}
%and
\begin{equation}
\label{eq:x1}
x_1:= -\frac{2(r_2\mu_2+{\mu_1})} {1+r_2^2+2\rho r_2}<0
\quad \text{and}\quad
%y_1:=
%,
%\end{equation}
%\begin{equation}
%\label{eq:y2}
y_2:= -\frac{2(r_1 {\mu_1}+{\mu_2})} {1+r_1^2+2\rho r_1}<0.
%\quad \text{and}\quad
%x_2:=
%.
\end{equation}
These points are represented on Figure~\ref{fig:kernelBrownian} and satisfy the following:
\begin{itemize}
\item $K(x_0,0)=k_2(x_0,0)=0$,  $K(0,y_0)=k_1(0,y_0)=0$. 
\item $\exists y_1 \in\mathbb{R}$ such that $K(x_1,y_1)=k_2(x_1,y_1)=0$\rems{.}
\item $\exists x_2\in\mathbb{R}$ such that  $K(x_2,y_2)=k_1(x_2,y_2)=0$.
\end{itemize}
%Note that for $\rho=0$ we also have $\psi(x_0,y_0)=0$.

\begin{figure}[h]
\centering
\includegraphics[scale=0.9]{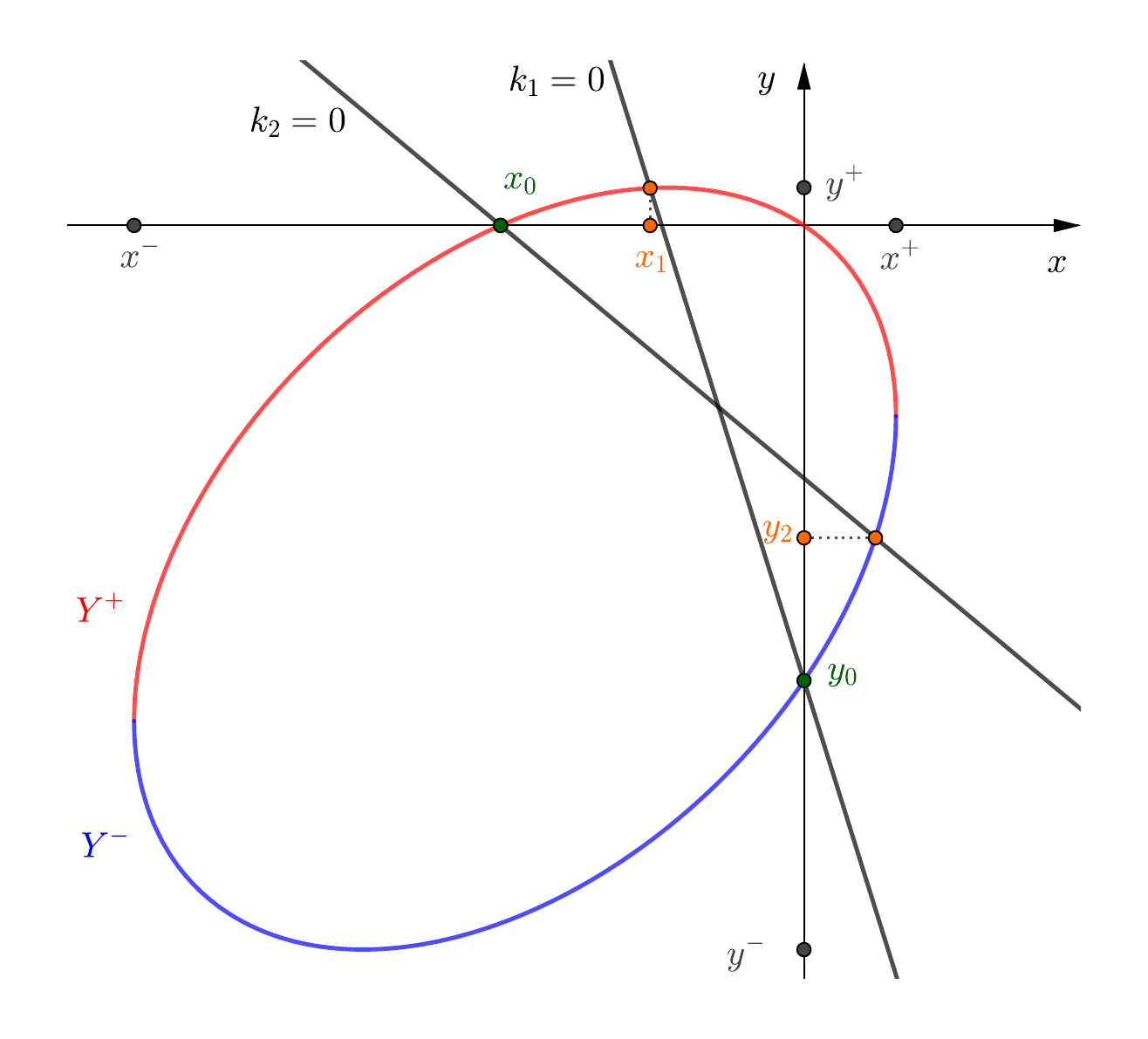}
%\vspace{-0.8cm}
\caption{The ellipse $K=0$,  the function $Y^-$ in blue, the function $Y^+$ in red, the two lines $k_1=0$ and $k_2=0$, the branch points $x^\pm$ and $y^\pm$, the points $x_0$ and $y_0$ in green, the points $x_1$ and $y_2$ in orange. This figure is drawn for the following parameters: $\mu_1=2$, 
$\mu_2=3$, $\rho=-0.4$, $r_1=2$, $r_2=4$.}
\label{fig:kernelBrownian}
\end{figure}
\noindent Let us define the curve $\mathcal{H}$, which is the boundary of the BVP established in Section~\ref{subsec:BVP}
\begin{equation}
\mathcal{H}= X^\pm([y^+,\infty))=\{x\in\mbC\colon K(x,y)=0 \text{ and } y\in [y^+,\infty)\}.
\label{eq:defHyp}
\end{equation}

\begin{lemma}[Hyperbola $\mathcal{H}$]
\label{lem:hyperbole}
The curve $\mathcal{H}$ is a branch of the hyperbola of equation
\begin{equation}
\label{eq:hyperbole}
(\rho^2-1)x^2+\rho^2 y^2- 2(\mu_1-\rho\mu_2)x= \mu_1(\mu_1-2\rho\mu_2).
\end{equation}
The curve $\mathcal{H}$ is symmetrical with respect to the horizontal axis and is the right branch of the hyperbola if $\rho<0$. Further, it is the left branch if $\rho>0$ and it is a straight line if $\rho=0$.
\end{lemma}

\begin{proof}
A similar kernel has already been studied; we refer the reader to \cite[Lemma~4]{franceschi_2019} and~\cite[Lemma~9]{baccelli_analysis_1987}, where the equation of such a hyperbola is derived.
\end{proof}

Let $\mathcal{H}^+$ denote the part of the hyperbola $\mathcal{H}$ with positive imaginary part. We also define the domain $\mathcal{G}$ bounded by $\mathcal{H}$ and containing $x^+$. This is depicted in Figure~\ref{fig:hyperbole} below.
\begin{figure}[hbtp]
\centering
\begin{subfigure}{0.32\textwidth}
\includegraphics[scale=1]{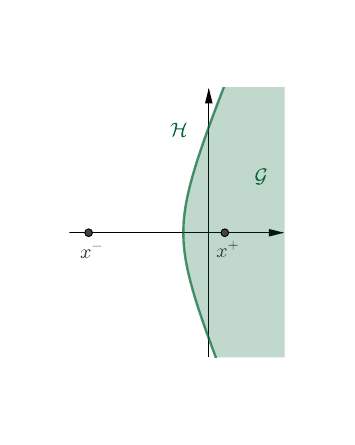}
\caption{$\rho<0$}
\label{fig:gull2}
\end{subfigure}
\begin{subfigure}{0.32\textwidth}
\includegraphics[scale=1]{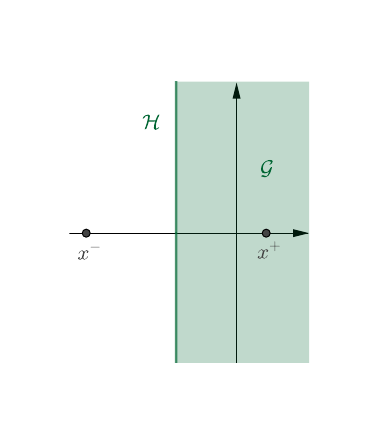}
\caption{$\rho=0$}
\label{fig:tiger2}
\end{subfigure}
\begin{subfigure}{0.32\textwidth}
\includegraphics[scale=1]{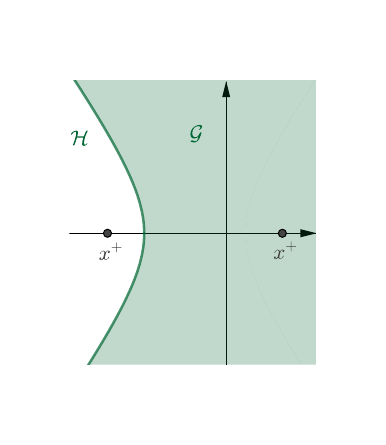}
\caption{$\rho>0$}
\label{fig:mouse2}
\end{subfigure}
\caption{Hyperbola $\mathcal{H}$ and domain $\mathcal{G}$.}
\label{fig:hyperbole}
\end{figure}

\subsection{Meromorphic continuation}

This section focuses on establishing the boundary value problem. We begin by meromorphically continuing the Laplace transform $\psi_1(x)$ (which converges for $x>0$).

\begin{lemma}[Meromorphic continuation]
\label{lem:continuationBrown}
By the formula
\begin{equation}
\label{eq:continuation}
\psi_1(x)=\frac{-k_2(x,Y^+(x))\psi_2(Y^+(x))}{k_1(x,Y^+(x))},
\end{equation}
the Laplace transform $\psi_1(x)$ can be meromorphically continued to the set
\begin{equation}
\label{eq:domainCV}
S:=\{x\in\mbC:\Re x>0 \text{ or } \Re Y^+(x)>0\} \cup\{ 0 \},
\end{equation}
where the domain $\mathcal{G}$ and its boundary $\mathcal{H}$ are included in the set defined in~\eqref{eq:domainCV}. Then $\psi_1$ is meromorphic on $\mathcal{G}$ and is continuous on $\overline{\mathcal{G}}$.
\end{lemma}

\begin{proof}
The Laplace transforms $\psi_1(x)$ and $\psi_2(y)$ are analytic, respectively, on $\{x\in\mbC\colon \Re x >0\}$ and $\{y\in\mbC\colon \Re y >0\}$. The functional equation~\eqref{eq:functionalequation} implies that for $(x,y)$ in the set $\widetilde{S}:=\{(x,y)\in\mathbb{C}^2: \Re x>0 , \ \Re y >0 \text{ and }  \psi(x,y)=0  \}$, we have
\begin{equation}
0=k_1(x,y)\psi_1(x)+k_2(x,y)\psi_2(y).
\label{eq:eqlemcontinuation}
\end{equation}
The open connected set
$$
S_1:=\{x\in\mbC\colon \Re Y^+(x)>0\},
$$
intersects the open set $S_2:=\{x\in\mbC: \Re x>0\}$. For $x\in S_1 \cap S_2$,  $(x,Y^+(x))\in \widetilde{S}$; equation~\eqref{eq:eqlemcontinuation} implies that the continuation formula in~\eqref{eq:continuation} is satisfied for all $x\in S_1 \cap S_2$. Figure~\ref{fig:domainS} represent\rems{s} these sets. With $\psi_1(x)$ defined as in~\eqref{eq:continuation}, we invoke the principle of analytic continuation and meromorphically extend $\psi_1$ to $S=S_1\cup S_2$. Note that the inclusion of $\mathcal{G}$ in the set $S$ defined in~\eqref{eq:domainCV} is similar to that in \cite[Lemma~5]{franceschi_2019}. This inclusion is depicted below in Figure~\ref{fig:domainS}.
\end{proof}

\begin{figure}[h]
\centering
\includegraphics[scale=0.34]{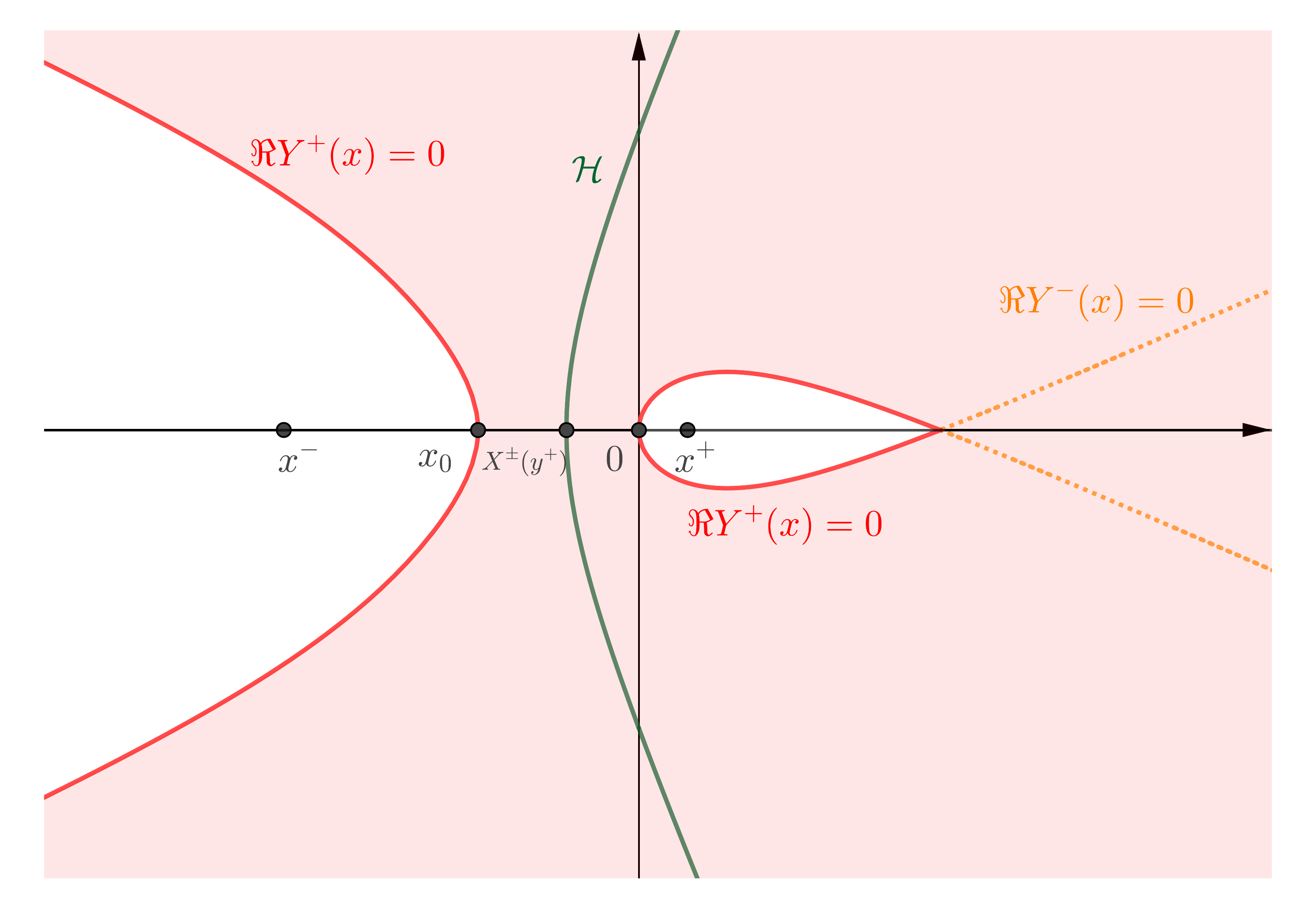}
%\vspace{-0.8cm}
\caption{The complex plane of $x$. The red curve of equation $\Re Y^+ (x)=0$ bounds the red domain $S_1:=\{x\in\mathbb{C}:\Re Y^+ (x) >0 \}$. The orange dotted curve corresponds to the equation $\Re Y^- (x)=0$. The domain $\mathcal{G}$ is bounded on the left by the green hyperbola $\mathcal{H}$, contains $x^+$ (see Figure~\ref{fig:hyperbole}), and is included in $S=S_1\cup S_2$, where $S_2:=\{x\in\mathbb{C}:\Re x >0 \}$. This figure is drawn for the parameters $\mu_1=2$, 
$\mu_2=3$, $\rho=-0.4$.}
\label{fig:domainS}
\end{figure}

\subsection{Poles and geometric conditions}

\begin{lemma}[Poles]
\label{lem:poleBrownian}
On the set $S$ defined in~\eqref{eq:domainCV}, the Laplace transform $\psi_1$ has either one or two poles, as follows:
\begin{itemize}
\item
(One pole:) If $k_1(x^-,Y^\pm(x^-))\geqslant 0$, the point $0$ is the unique pole of $\psi_1$ in $S$ and this pole is simple.
\item
(Two poles:) If $k_1(x^-,Y^\pm(x^-))<0$, the points $0$ and $x_1$ (defined in \eqref{eq:x1}) are the only possible poles of $\psi_1$ in $S$ and these poles are simple; $x_1\in S$ if and only if $x_1>x_0$.
\end{itemize}
In addition, $\lim_{x\to 0} x \psi_1(x)=1.$ Further, the point $x_1$ is a pole of $\psi_1$ and belongs to the domain $\mathcal{G}$ if and only if
$k_1(X^\pm(y^+),y^+) < 0$.
\end{lemma}

\begin{proof}
The final value theorem for the Laplace transform, together with Theorem~\ref{thm:uvtendtoinfty}, imply that
$$\lim_{x\to 0} x \psi_1(x)= \lim_{u\to\infty} \mathbb{P}_{(u,0)} [T=\infty ] =1.$$ We may thus conclude that $0$ is a simple pole.
On the set $\{x \in\mathbb{C}: \Re x>0  \}$, $\psi_1$ is defined as a Laplace transform which converges (and thus has no poles).
Therefore, with the exception of $0$, the only possible poles in $S$ are the zeros of $k_1(x,Y^+(x))$, which are the zeros of the denominator of the continuation formula in~\eqref{eq:continuation}. Straightforward calculations show that equation $k_1(x,Y^+(x))=0$ has either no roots or one (simple) root, and that this depends on the sign of $k_1(x^-,Y^\pm(x^-))$. When the root exists, it is $x_1$ (see \eqref{eq:x1}). The condition for the existence of this root is depicted in Figure~\ref{fig:rootcondition} below. It now only remains to remark that when $x_1$ is a pole, $x_1$ is in $\mathcal{G}$ if and only if $x_1 > X^\pm (y^+)$. The latter holds if and only if $k_1(X^\pm(y^+),y^+) < 0$ (see Figure~\ref{fig:rootconditionbis}).
\end{proof}

\begin{figure}[h]
\centering
\includegraphics[trim={0cm 0.4cm 0cm 0.5cm}, clip,scale=0.63]{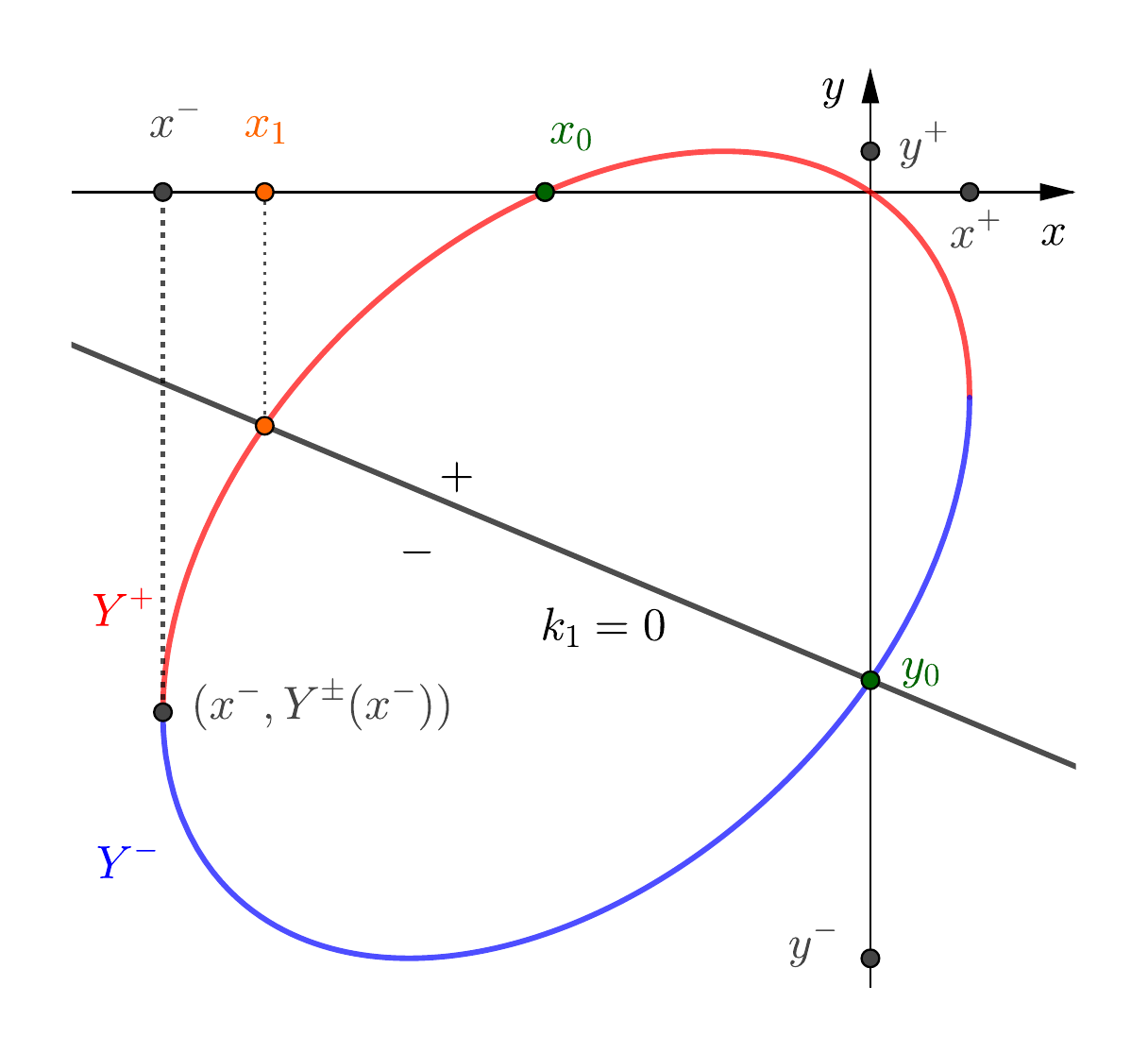}
\includegraphics[trim={0cm 0.4cm 0cm 0.5cm}, clip,scale=0.63]{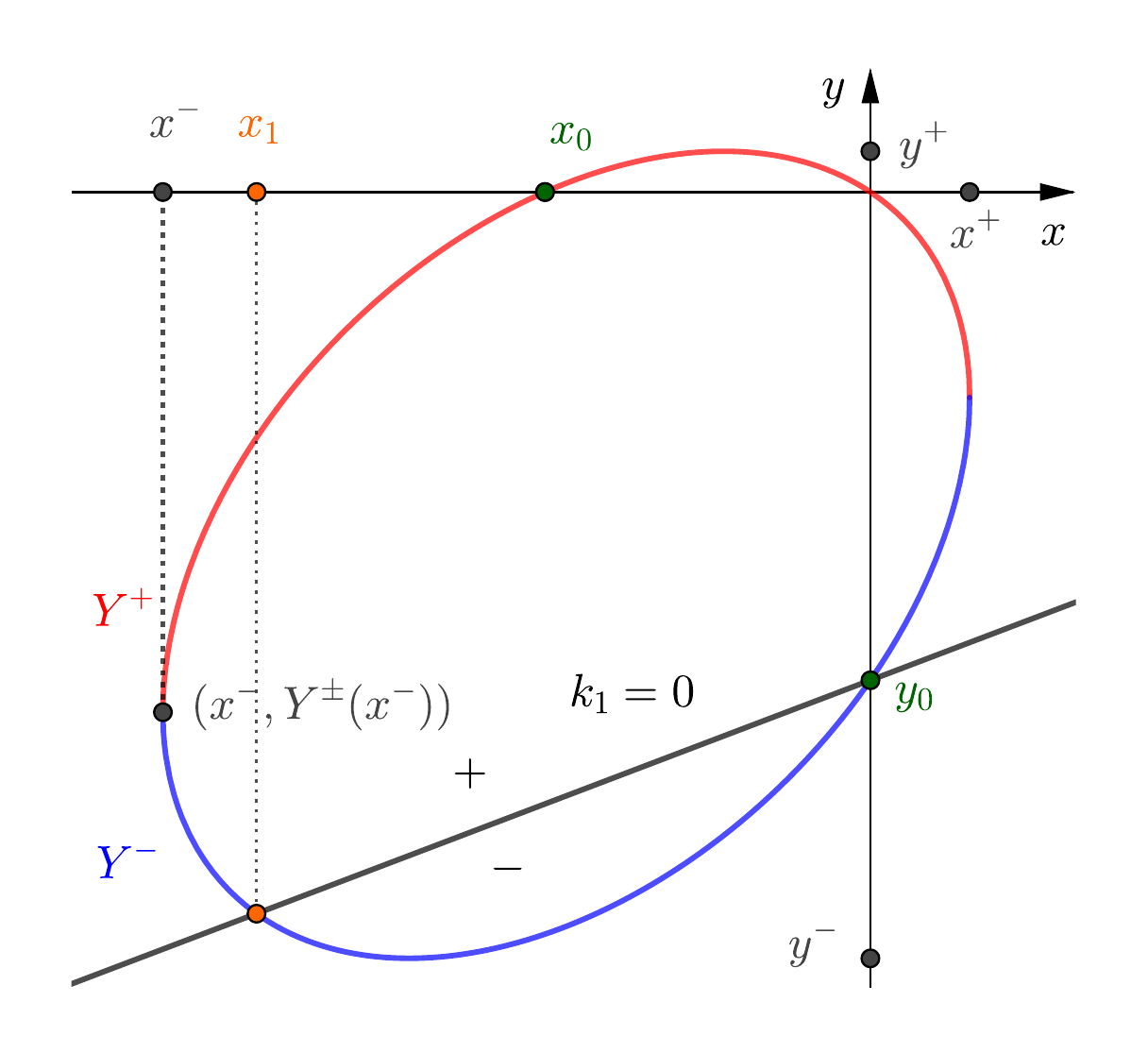}
%\vspace{-0.8cm}
\caption{On the left, we see that $k_1(x^-,Y^\pm(x_-))<0$ and $x_1$ is a simple pole of $\psi_1$. On the right, we see that $k_1(x^-,Y^\pm(x^-))>0$ and $\psi_1$ has no pole in $S$.}
\label{fig:rootcondition}
\end{figure}

\begin{figure}[h]
\centering
\includegraphics[trim={0cm 0.5cm 0cm 0.5cm}, clip,scale=0.63]{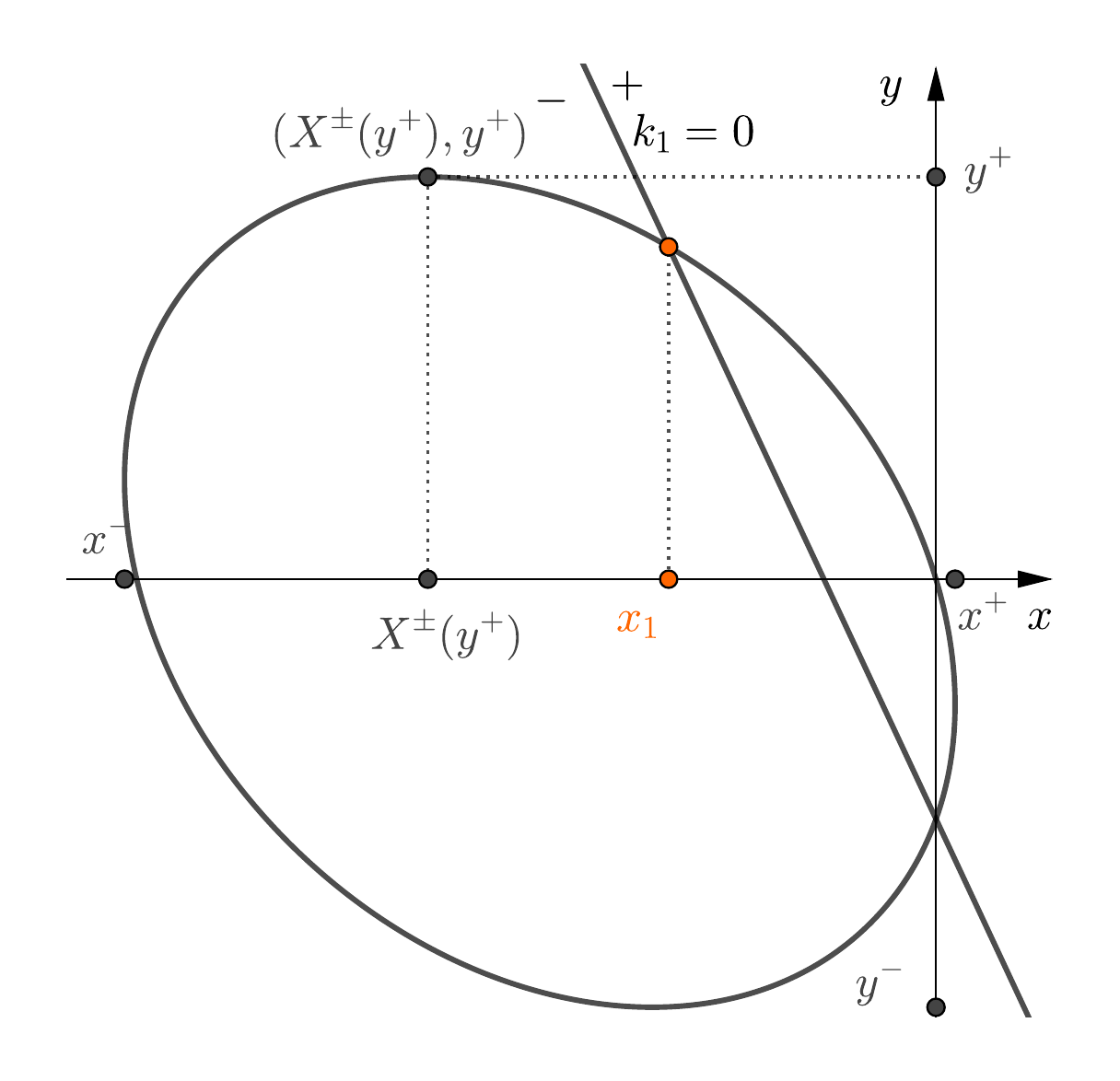}
\includegraphics[trim={0cm 0.5cm 0cm 0.5cm}, clip,scale=0.63]{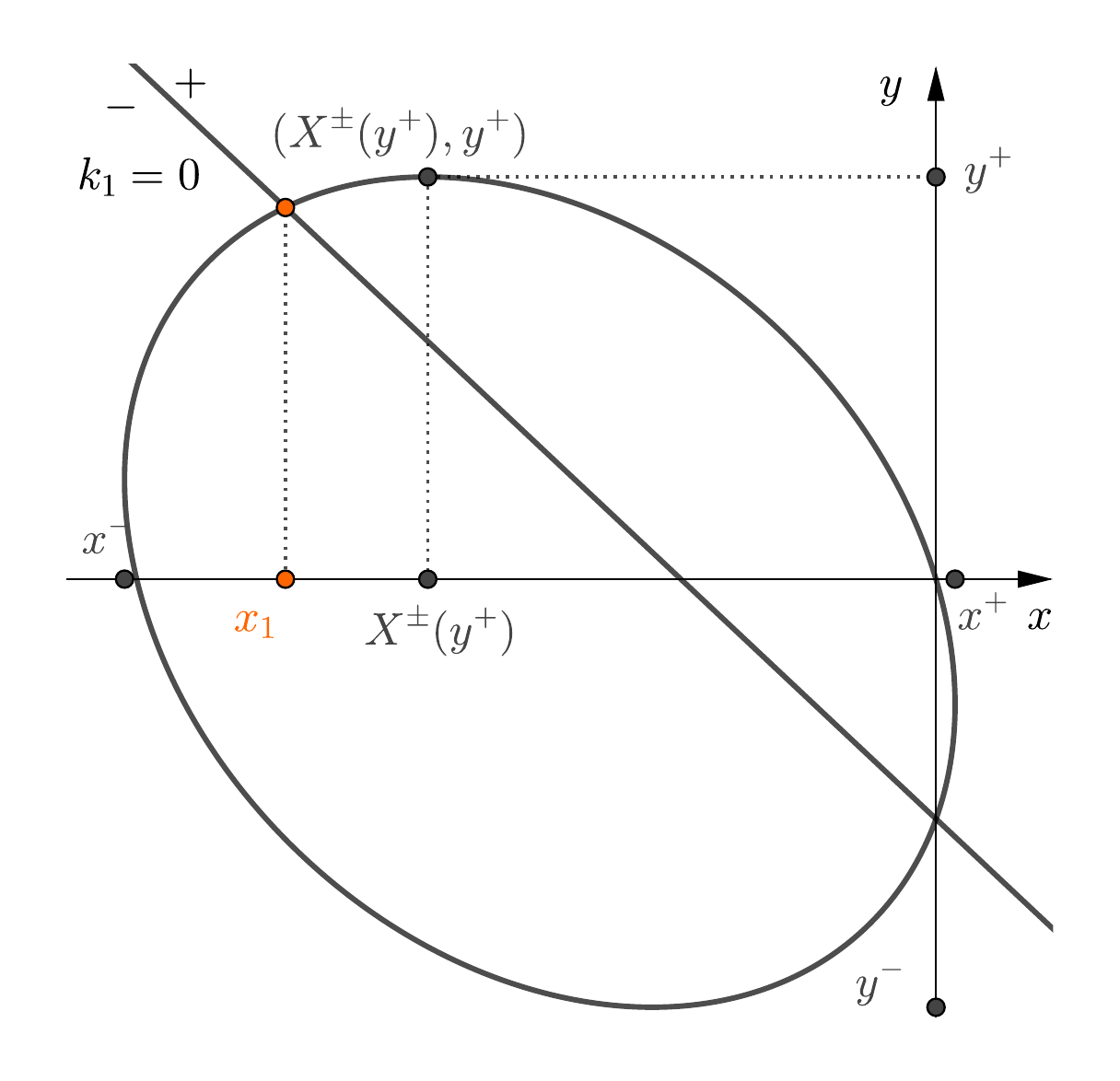}
%\vspace{-0.8cm}
\caption{On the left, we see that $k_1(X^\pm(y^+),y^+) < 0$ and $x_1$ is in $\mathcal{G}$. On the right, we see that  $k_1(X^\pm(y^+),y^+) > 0$ and $x_1$ is not in $\mathcal{G}$.}
\label{fig:rootconditionbis}
\end{figure}

Before turning to Lemma \ref{lem:condgeom}, recall that the angles $\delta$, $\beta$ and $\theta$ were defined above in \eqref{eq:defangles} and that $k_1$ was defined in \eqref{eq:k1k2}.
\begin{lemma}[Geometric conditions]
The condition $k_1(x^-,Y^\pm(x^-))> 0$ (resp. $=0$ and $<0$) is equivalent to 
$$
2\delta -\theta < \pi\rems{,}
$$
(resp. $=\pi$ and $> \pi$). 
The condition $k_1(X^\pm(y^+),y^+) > 0$ (resp. $=0$ and $<0$) is equivalent to 
$$
2\delta -\theta +\beta< 2\pi\rems{,}
$$
(resp. $=2\pi$ and $>2\pi$).
\label{lem:condgeom}
\end{lemma}
\begin{proof}
By condition \eqref{eq:conditionRmatrix} and by the fact that the drift is positive, we have $0<\theta<\beta<\delta<\pi$. By \eqref{eq:defangles} and \eqref{eq:defxypm}, 
\begin{equation}
x^-/\mu_2 = \frac{1}{\sqrt{1-\rho^2}} \left(
\frac{\rho-\mu_1/\mu_2}{\sqrt{1-\rho^2}}
- \sqrt{\left( \frac{\rho-\mu_1/\mu_2}{\sqrt{1-\rho^2}} \right)^2  +1 } 
\right)
=\frac{-\cot(\theta)
- \sqrt{\cot^2 (\theta)  +1 } }{\sin(\beta)}.
\label{eq:mu2x}
\end{equation}
We begin by proving the first equivalence for $\delta\geqslant \pi/2$. In this case we have
\begin{align*}
k_1(x^-,Y^\pm(x^-))> 0
& \Leftrightarrow
\frac{1}{2}(r_2x^- +Y^\pm(x^-))+\rho x^-+\mu_2>0
\\ & \Leftrightarrow 
%(r_2+\rho)x^- +\mu_2>0
r_2+\rho < -\mu_2/x^- \text{ since } Y^\pm(x^-)=-\rho x^- -\mu_2 \text{ by } \eqref{eq:defXYpm} \text{ and } \eqref{eq:defxypm}
\\ & \Leftrightarrow 
r_2-\cos(\beta)<  \sin (\beta) \left(\cot(\theta)+\sqrt{\cot^2(\theta)+1}\right)^{-1} \text{ by } \eqref{eq:mu2x}
\\ & \Leftrightarrow
-\cot(\delta) \left(\cot(\theta)+\sqrt{\cot^2(\theta)+1}\right) <1
\\ & \Leftrightarrow
0<-\cot(\delta)\sqrt{\cot^2(\theta)+1}<1+\cot(\delta)\cot(\theta)     \text{ since we assumed } \delta \geqslant \pi/2
\\ & \Leftrightarrow
\cot^2(\delta)({\cot^2(\theta)+1})<(1+\cot(\delta)\cot(\theta)   )^2
\\ & \Leftrightarrow
2\cot(\delta)\cot (\theta) - \cot^2 (\delta)+1  > 0 
\\ & \Leftrightarrow
2\sin(\delta)\cos(\delta)\cos (\theta) - (\cos^2 (\delta)-\sin^2 (\delta)) \sin(\theta) > 0 
\\ & \Leftrightarrow
\sin(2\delta)\cos (\theta) - \cos (2\delta) \sin(\theta) > 0 
\\ & \Leftrightarrow
\sin(2\delta -\theta) > 0 
\\ & \Leftrightarrow 
2\delta -\theta < \pi \text{ since } 0<2\delta -\theta< 2\pi.
\end{align*}
It is straightforward to see that if $\delta<\pi/2$, then $2\delta -\theta < \pi $. Further, $\delta<\pi/2$ is equivalent to $r_2+\rho <0$ by \eqref{eq:defangles}, which implies that $r_2+\rho < -\mu_2/x^- $. Therefore, $k_1(X^\pm(y^+),y^+) < 0$. This proves the first equivalence. The second equivalence is proved in exactly the same way, so the details are omitted. This concludes the proof.
\end{proof}

\subsection{Absorption \rems{a}symptotics along the axes}
In this section, we establish asymptotic\remst{s} results for the absorption probability (and escape probability) in a simpler case where the starting point is $(u,0)$.
\begin{proposition}[Absorption asymptotics]
\label{cor:asymptBrown}
Let us assume that $x^- \in S$. For some constant $C$, the asymptotic behavior of $\mathbb{P}_{(u,0)}[T<\infty]$ as $u\to\infty$ is given by
$$
\mathbb{P}_{(u,0)}[T<\infty]\sim C
\begin{cases}
e^{ux_1} & \text{ if } %k_1(x^-,Y^\pm(x^-))<0,
2\delta -\theta > \pi, \\
u^{-\frac{3}{2}}e^{ux^-} & \text{ if } %k_1(x^-,Y^\pm(x^-))>0,
2\delta -\theta < \pi, \\
u^{-\frac{1}{2}}e^{ux^-} & \text{ if } %k_1(x^-,Y^\pm(x^-))=0.
2\delta -\theta = \pi.
\end{cases}
$$
\end{proposition}

\begin{proof}
The largest singularity of the Laplace transform of $\mathbb{P}_{(u,0)}[T<\infty]$ determines its asymptotics. We proceed by invoking a classical transfer theorem, see ~\cite[Theorem~37.1]{doetsch_introduction_1974}. This theorem says that if $a$ is the largest singularity of order $k$ of the Laplace transform (that is, the Laplace transform behaves as $(s-a)^{-k}$ up to additive and multiplicative constants in the neighborhood of $a$), then when $u\to\infty$, the probability $\mathbb{P}_{(u,0)}[T<\infty]$ is equivalent (up to a constant) to $u^{k-1}e^{au}$. The Laplace transform of $\mathbb{P}_{(u,0)}[T<\infty]$ is $1/x-\psi_1(x)$. By Lemma~\ref{lem:poleBrownian}, the point $0$ is not a singularity and the point $x_1$ is a simple pole. %and the largest singularity of $\psi_1$ if $k_1(x^-,Y^\pm(x^-))<0$. 
When that pole exists, the asymptotics are given by $Ce^{u x_1}$ for some constant $C$. When there is no pole, that is, when $k_1(x^-,Y^\pm(x^-))\ges 0$, the largest singularity is given by the branch point $x^-$. The definition of $Y^+$ and ~\eqref{eq:continuation} together imply that for some constants $C_i$ we have
\[
\psi_1(x)\underset{x\to x^-}{=}
\begin{cases}
C_1+C_2\sqrt{x-x^-}+\Oh(x-x^-) & \text{ if } k_1(x^-,Y^\pm(x^-))>0,\\
\dfrac{C_3}{\sqrt{x-x^-}}+\Oh(1) & \text{ if } k_1(x^-,Y^\pm(x^-))=0.
\end{cases}
\]
The proof is then completed by applying  Lemma~\ref{lem:condgeom} and invoking the classical transfer theorem.
\end{proof}

\begin{remark}[Asymptotics along the vertical axis]
Studying the singularities of $\phi_1$ we obtained in Proposition~\ref{cor:asymptBrown} the asymptotics of the absorption probability (and then of the escape probability which is equal to $1-\mathbb{P}_{(u,0)}[T<\infty]$) along the horizontal axis. A similar study for $\psi_2$ would lead to the following asymptotics along the vertical axis. As $v\to\infty$
$$
\mathbb{P}_{(0,v)}[T<\infty]\sim C
\begin{cases}
e^{vy_2} & \text{ if } %k_1(x^-,Y^\pm(x^-))<0,
2\epsilon +\theta-\beta  > \pi, \\
v^{-\frac{3}{2}}e^{vy^-} & \text{ if } %k_1(x^-,Y^\pm(x^-))>0,
2\epsilon +\theta-\beta < \pi, \\
v^{-\frac{1}{2}}e^{vy^-} & \text{ if } %k_1(x^-,Y^\pm(x^-))=0.
2\epsilon +\theta-\beta = \pi.
\end{cases}
$$
\end{remark}

\rems{\begin{remark}[Bivariate asymptotics]
The bivariate asymptotics of the absorption probability could be derived using the saddle point method and studying the singularities, see \cite{Franceschi2017} and \cite{ernst_2020}. Such a study is very technical and requires to distinguish a lot of different cases. We would obtain some functions $a$, $b$, $c$ depending on the parameters, such that for $(u,v)=(r\cos(t),r\sin(t))$ in polar coordinates, 
$$
\mathbb{P}_{(u,v)}[T<\infty]\underset{r\to\infty}{\sim} a(t) r^{b(t)} e^{-c(t) r}.
$$
Typically $b$ would take the values $0$ or $-1/2$.
\end{remark}}

%In the case where $x^-$ is not in $S$ some asymptotics results also hold. But to show it we should extend again the Laplace transform to a bigger domain than $S$ and some other poles could appear which makes the task more technical and complicated.

\section{Product form and exponential absorption probability}
\label{sec:productform}

In this section, we consider a remarkable geometric condition on the parameters characterizing the case where the absorption probability has a product form and is exponential. We call this new criterion the \textit{dual skew symmetry} condition due to its natural connection with the famous skew symmetry condition studied by Harrison, Reiman and Williams \cite{harrison_reiman_1981,harrison_multidimensional_1987}, which characterizes the cases where the stationary distribution has a product form and is exponential. The \textit{dual skew symmetry} condition  gives a criterion for the solution to the partial differential equation of Proposition~\ref{prop:PDE} (dual to that satisfied by the invariant measure) to be of product form. The following Theorem state\rems{s} a simple geometric criterion on the parameters for the absorption probability to be of product form; the absorption probability is then exponential. 

\begin{theorem}[Dual skew symmetry]
Let $f(u,v)=\mathbb{P}_{(u,v)} [T<\infty]$ be the absorption probability.
The following statement\rems{s} are equivalent:
\begin{enumerate}
\item \label{item11} The absorption probability has a product form, i.e. \rems{there} exi\rems{s}t\remst{s} $f_1$ and $f_2$ such that $$f(u,v)=f_1(u)f_2(v) ;$$
\item \label{item22} The absorption probability is exponential, i.e. there exist\remst{s} $x$ and $y$ in $\mathbb{R}$ such that $$f(u,v)=e^{ux+vy};$$
\item \label{item33} The reflection vectors are in opposite directions, i.e. 
$$r_1r_2=1 ;$$
\item \label{item44} The reflection angles in the wedge satisfy $\alpha=1$, i.e.
$$\delta+\epsilon-\beta=\pi.$$
\end{enumerate}
In this case we have
$$
f(u,v)= e^{ux_1+vy_2}
$$
where $x_1$ and $y_2$ are given in \eqref{eq:x1}.
\label{thm:productform}
\end{theorem}
\begin{proof}
\eqref{item11} $\Rightarrow$ \eqref{item22}:
Neumann boundary conditions \eqref{eq:Neumanboundcond} imply that
$f_1'(0)f_2(y)-r_1f_1(0)f_2'(y)=0$ and $-r_2 f_1'(u)f_2(0)+f_1(u)f_2'(0)=0$. Solving these differential equation\rems{s} imply that $f_1$ and $f_2$ (and thus $f$) are exponential.
\\
\eqref{item22} $\Rightarrow$ \eqref{item11}: This implication is straightforward.
\\
\eqref{item22} $\Rightarrow$ \eqref{item33}:
Neumann boundary conditions \eqref{eq:Neumanboundcond} imply that for all $v>0$, $xe^{vy} - r_1 y e^{vy}=0$  and that for all $u>0$, $-r_2 x e^{ux} +ye^{ux}=0$. It follows that $r_1=x/y$, $r_2=y/x$, and thus $r_1r_2=1$.
\\
\eqref{item33} $\Rightarrow$ \eqref{item22}:
Let us define $f(u,v)= e^{ux_1+vy_2}$. We need to show that $f$ satisfies the partial differential equation of Proposition~\ref{prop:PDE}. This will imply that $f$ is the absorption probability.
The fact that $r_1=1/r_2$, combined with \eqref{eq:x1}, gives $r_1=x_1/y_2$. This implies that $f$ satisfies the Neumann boundary conditions  in \eqref{eq:Neumanboundcond}. The limit values are satisfied because $f(0,0)=1$ and $\lim_{(u,v)\to\infty} f(u,v)=0$ for $x_1<0$ and $y_2<0$. It now only remains to show that $\mathcal{G}f=0$. We now only need verify that $K(x_1,y_2)=0$, see Figure~\ref{fig:productform}. By definition of $y_2$ (see \eqref{eq:x1}), we have
\begin{align*}
K(x_1,y_2)&=y_2\left( \frac{y_2 }{2}\left( \left(\frac{x_1}{y_2} \right)^2+1+2\rho \frac{x_1}{y_2} \right) +\mu_1 \frac{x_1}{y_2}+\mu_2 \right)
\\ &=
y_2\left(  \frac{y_2}{2}\left( r_1^2+1+2\rho r_1 \right) +\mu_1 r_1+\mu_2 \right)=0.\\
\end{align*}
\eqref{item33} $\Leftrightarrow$ \eqref{item44}: The following equivalences hold:
\begin{align*}
r_1r_2=1 & \Leftrightarrow \left( \sin(\beta)/\tan(\delta)-\cos(\beta) \right) \left( \sin(\beta)/\tan(\epsilon)-\cos(\beta) \right)=1 \quad \text{by } \eqref{eq:defangles}
\\ & \Leftrightarrow  \frac{\sin(\beta)}{\tan(\epsilon}=\frac{\tan(\delta)}{\sin(\beta)-\cos(\beta)\tan(\delta)}+\cos(\beta)=\frac{\tan(\delta)(1-\cos^2(\beta)+\cos(\beta)\sin(\beta)}{\sin(\beta)-\cos(\beta)\tan(\delta)}
\\ & \Leftrightarrow \tan(\epsilon)=\frac{\tan(\beta)-\tan(\delta)}{1+\tan(\delta)\tan(\beta)}
\\ & \Leftrightarrow \tan(\epsilon)=\tan (\beta-\delta)
\\ & \Leftrightarrow \epsilon =\beta-\delta +\pi.
\end{align*}
\end{proof}

%\paragraph*{\textbf{Remark}}
\begin{remark}[Standard and dual skew symmetry] 
The standard skew symmetry condition for the matrix $\left(\begin{array}{cc}
1 & -r_2 \\ 
-r_1 & 1
\end{array} \right) $ is $2\rho=-r_1-r_2$ or equivalently $\epsilon+\delta=\pi$. The standard skew symmetry condition for the dual matrix $\left(\begin{array}{cc}
r_2 & -1 \\ 
-1 & r_1
\end{array}\right)$ defined in Section~\ref{subsec:dualprocess} is $2\rho=-1/r_1-1/r_2$ or equivalently $\epsilon+\delta-2\beta=\pi$. %Note that the condition $r_1r_2=1$ is independent of the covariance $\rho$; however, the condition $\delta+\epsilon-\beta=\pi$ involves the wedge angle $\beta$. 
Note that the dual skew symmetry condition obtained in Theorem~\ref{thm:productform} is different from these two conditions.
Further properties of the \textit{dual skew symmetry} condition will be explored in future work. 
\end{remark}

%As is the case for skew symmetry, we believe that such a remarkable result remain valid in higher dimensions  Also remaining in a two-dimensional wedge, we believe that a condition characterizing the cases where the absorption probability is a sum of exponential should exist and would be related to the result obtain by \citet{dieker_reflected_2009} for the stationary distribution. We will explore it in future works. 

\begin{figure}[h]
\centering
\includegraphics[trim={0cm 0.5cm 0cm 0.5cm}, clip,scale=1.1]{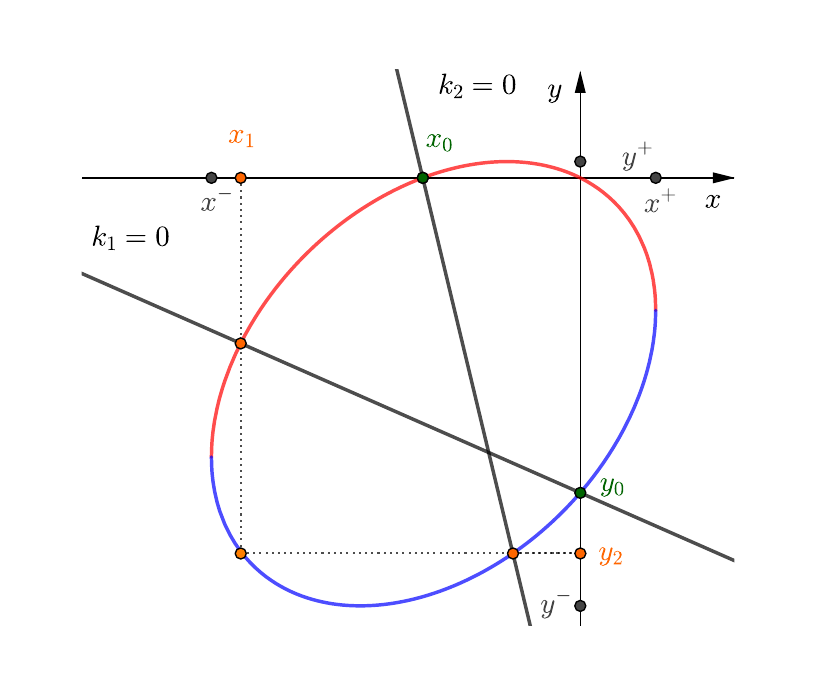}
\includegraphics[trim={0cm 0.5cm 0cm 0.5cm}, clip,scale=0.6]{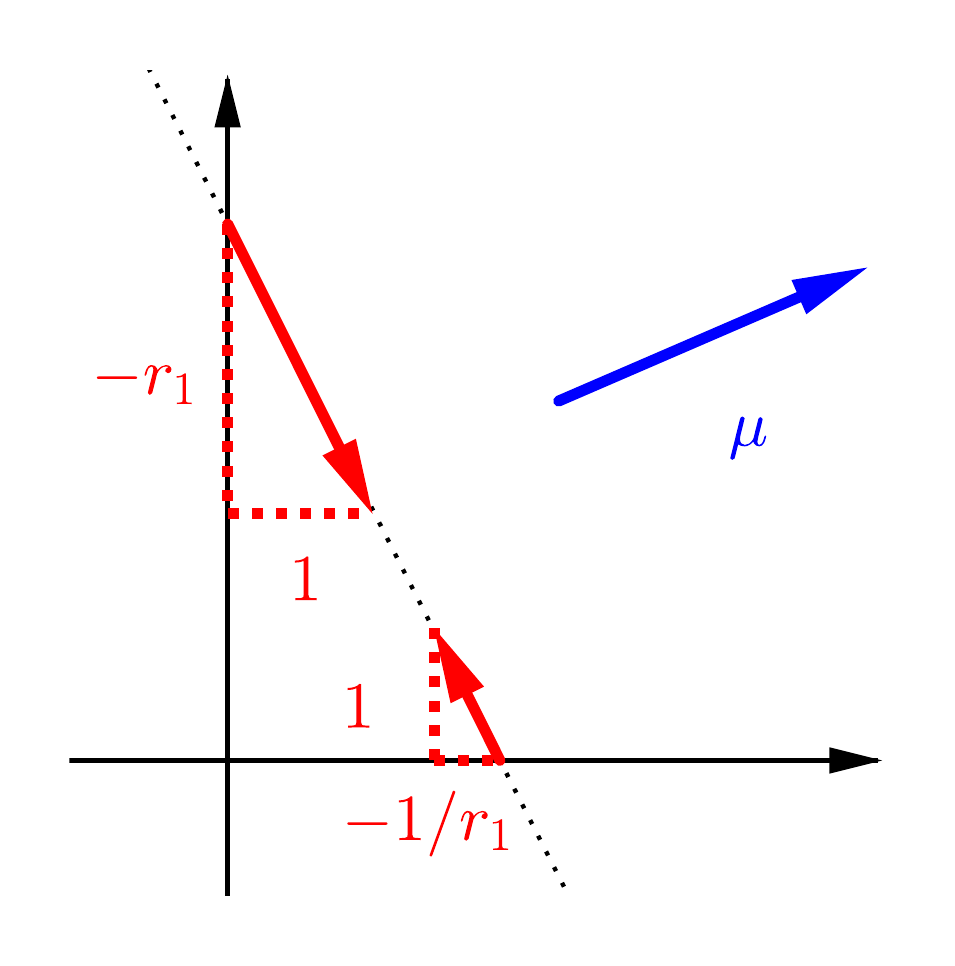}
%\vspace{-0.8cm}
\caption{Dual skew symmetry: on the left, we see that $K(x_2,y_2) = 0$; on the right, we see that condition $r_1r_2=1$ impl\rems{ies} that the reflection vectors are in opposite directions.}
\label{fig:productform}
\end{figure}

\section{Integral expression of the Laplace transform $\psi_1$}
\label{sec:integralexpression}

In this section, we establish  a boundary value problem (BVP) satisfied by the Laplace transform (Proposition~\ref{prop:BVPBrownian}). The section's key result is Theorem~\ref{thm:explicitF1Brown}, which gives an explicit integral formula for the Laplace transform of the escape probability. 

\subsection{Carleman boundary value problem}
\label{subsec:BVP}

We state a Carleman BVP satisfied by the Laplace transform $\psi_1$.

\begin{proposition}[Carleman BVP]
\label{prop:BVPBrownian}
The Laplace transform $\psi_1$ satisfies the following boundary value problem:
\begin{enumerate}[label=(\roman*)]
\item
\label{item:1a}
$\psi_1(x)$ is meromorphic on $\mcG$ and continuous on $\overline{\mathcal{G}}$.
\item
\label{item:1abis}
$\psi_1(x)$ admits one or two poles in $\mathcal{G}$. $0$ is always a simple pole and $x_1$ is a simple pole if and only if 
$2\delta -\theta +\beta> 2\pi$.  %$k_1(X^\pm(y^+),y^+) < 0$ ;
%$x_1>X^+(y^+)$ and $k_1(x^-,Y^\pm(x^-))<0$;
\item
\label{item:2a}
$\lim_{x\to\infty} x \psi_1(x)=0$.
\item
\label{item:3a}
$\psi_1$ satisfies the boundary condition
$$
\psi_1(\overline{x})=G(x)\psi_1(x),\quad \forall\, x\in\mcH,
$$
where
\begin{equation}
\label{eq:Gdefbrownien}
G(x):=\frac{k_1}{k_2}(x,Y^+(x)) \frac{k_2}{k_1}(\overline{x},Y^+(x)).
\end{equation}
\end{enumerate}
\end{proposition}

\begin{proof}
Statement \ref{item:1a} immediately follows from Lemma~\ref{lem:continuationBrown}. Statement \ref{item:1abis} immediately follows from Lemmas~\ref{lem:poleBrownian} and \ref{lem:condgeom}. Statement ~\ref{item:2a} follows from the initial value theorem for the Laplace transform, which implies that $\lim_{x\to\infty} x \psi_1(x)=  \mathbb{P}_{(0,0)} [T=\infty ] =0$.
%$\psi_1$ and $\psi_2$ 
%(it is straigforward to see that these integrals tend to $0$ when $x$ or $y\to\infty$) and from the continuation formula~\eqref{eq:continuation}.
To prove statement~\ref{item:3a}, we recall the functional equation~\eqref{eq:functionalequation}.  For $x\in\mathcal{H}$, we evaluate this equation for $(x,Y^+(x))$ and $(\overline{x},Y^+(\overline{x}))$. By the definition of $Y^+$, we have $K(x,Y^+(x))=K(\overline{x},Y^+(\overline{x}))=0$. By the definition of the hyperbola $\mathcal{H}$ in \eqref{eq:hyperbole}, we have $Y^+(\overline{x})=Y^+({x})$. This enables us to obtain the following system of equations
$$
\begin{cases}
0=k_1(x,Y^+(x)) \psi_1 (x)+
k_2(x,Y^+(x)) \psi_2 (Y^+(x)),
\\ 
0=k_1(\overline{x},Y^+(x)) \psi_1 (\overline{x})+
k_2(\overline{x},Y^+(x)) \psi_2 (Y^+(x)).
\end{cases}
$$
Solving this system of equations and eliminating $\psi_2 (Y^+(x))$, we obtain the boundary condition in statement~\ref{item:3a}.
\end{proof}

\subsection{Gluing function}

To solve the BVP, we need a conformal gluing function which glues together the upper and lower parts of the hyperbola. This conformal gluing function was introduced in~\cite{franceschi_tuttes_2016,franceschi_2019}. For $a\ges 0$ and for $x\in\mbC\setminus (-\infty,-1]$, the generalized Chebyshev polynomial is defined by
$$
T_a(x):=\cos(a\arccos(x))=\dfrac{1}{2}\left((x+\sqrt{x^2-1})^a+ (x-\sqrt{x^2-1})^a\right).
$$
We define the angle
$$
\beta:=\arccos(-\rho).
$$
We also define the functions
\begin{equation}
\label{eq:wdefBrown}
w(x)\coloneqq T_{\frac{\pi}{\beta}}\left(\dfrac{2x-(x^++x^-)}{x^+-x^-}\right),
\end{equation}
and
$$
W(x):= \frac{w(x)-w(X^\pm(y^+))}{w(x)-w(0)}.
$$
We now recall a useful \rems{l}emma from \cite{franceschi_2019} for the conformal gluing function $W$.
\begin{lemma}[Lemma 9, \cite{franceschi_2019}]
\label{lem:gluingwBrownien}
The function $W$
satisfies the following properties
\begin{enumerate}[label=(\roman*)]
\item
$W$ is holomorphic in $\mathcal{G}\setminus \{ 0 \}$, continuous in $\overline{\mathcal{G}}\setminus \{ 0 \}$ and bounded at infinity.
\item
$W$ is bijective from $\mathcal{G}\setminus \{ 0 \}$ to $\mathbb{C}\setminus [0,1]$.
%$\mathbb{C}\setminus (-\infty,-1]$;
\item
$W$ satisfies the gluing property on the hyperbola
$$
W(x)=W(\overline{x}),\quad \forall\, x\in\mathcal{H}.
$$
\end{enumerate}
\end{lemma}

%\subsection{Riemann boundary value problem}

\subsection{Index}

We begin with some necessary notation. Let the angle $\Delta$ be the variation of the argument of $G(x)$ when $x$ lies on $\mathcal{H^+}$:
$$
\Delta:=[\arg G(x)]_{\mathcal{H^+}}=  \left[\arg \frac{k_1}{k_2} (x,Y^+(x))\right]_{\mathcal{H}}.
$$
Further, let $d$ be the argument of $G$ at the real point of the hyperbola $\mathcal{H}$:
$$
d:= \arg G(X^+(y^+))\in(-\pi,\pi].
$$
We define the index $\kappa$ as
$$
\kappa:=\left\lfloor \frac{d+\Delta}{2\pi} \right\rfloor.
$$
The index shall prove useful to solving the boundary value problem given in Proposition~\ref{prop:BVPBrownian}.
\begin{lemma}
We have
$$
d=
\begin{cases}
0 & \text{if } k_1(x^-,Y^\pm(x^-))\neq 0 \text{ i.e. } 2\delta -\theta +\beta\neq 2\pi ,
\\
\pi & \text{if } k_1(x^-,Y^\pm(x^-))= 0 \text{ i.e. } 2\delta -\theta +\beta= 2\pi,
\end{cases}
$$
and %the angle $d+\Delta\in(-4\pi,2\pi)$ and satisfy
$$
\tan \frac{d+\Delta}{2}=\frac{(1-(r_1+2\rho)(r_2+2\rho))\sqrt{1-\rho^2}}{r_1+r_2+3\rho-r_1r_2\rho-2(r_1+r_2)\rho^2-4\rho^3} = \tan (\epsilon+\delta+\beta).
$$
Note also that $\epsilon+\delta+\beta\geqslant 2\pi$ is equivalent to $1-(r_1+2\rho)(r_2+2\rho)\leqslant 0$.
\label{lem:techniqueangle}
\end{lemma} 
\begin{proof}
The proof is in each step similar to the proof of \cite[Lemma 13]{franceschi_2019}. Firstly, note that the value of $d$ is obtained by the fact that $G(X^+(y^+))=1$ if $k_1(x^-,Y^\pm(x^-))\neq 0$ and that $G(X^+(y^+))=-1$ if $k_1(x^-,Y^\pm(x^-))= 0$.
We recall that by definition we have $\Delta = \lim_{x\to\infty \atop x\in \mathcal{H^+}} \arg G(x) -d $ and that by \eqref{eq:Gdefbrownien} we have
$$
G(x)=\frac{(\frac{1}{2}(r_2  +Y^+(x)/x)+\rho +\mu_2/x)(\frac{1}{2}(1+r_1 Y^+(x)/\overline{x})+\rho Y^+(x)/\overline{x}+\mu_1/\overline{x})}{(\frac{1}{2}(r_2  +Y^+(x)/\overline{x})+\rho +\mu_2/\overline{x})(\frac{1}{2}(1+r_1 Y^+(x)/{x})+\rho Y^+(x)/{x}+\mu_1/{x})}.
$$
By \eqref{eq:defXYpm}, we may compute the limit $$\lim_{x\to\infty \atop x\in \mathcal{H^+}} \frac{Y^+(x)}{x} =-\rho+i\sqrt{1-\rho^2},$$
from which we obtain
\begin{align*}
e^{i(\Delta+d)}&=\lim_{x\to\infty \atop x\in \mathcal{H^+}} G(x)
%\\ &=
%\frac{(\frac{1}{2}(r_2  -\rho+i\sqrt{1-\rho^2})+\rho )(\frac{1}{2}(1+r_1 (-\rho-i\sqrt{1-\rho^2}))+\rho (-\rho-i\sqrt{1-\rho^2}))}{(\frac{1}{2}(r_2  -\rho-i\sqrt{1-\rho^2})+\rho )(\frac{1}{2}(1+r_1(-\rho+i\sqrt{1-\rho^2}))+\rho (-\rho+i\sqrt{1-\rho^2}))}
\\ &=
\frac{(r_2  +\rho+i\sqrt{1-\rho^2}) (1-r_1 \rho-2\rho^2-i(r_1+2\rho)\sqrt{1-\rho^2})}{(r_2  +\rho-i\sqrt{1-\rho^2})(1-r_1 \rho-2\rho^2+i(r_1+2\rho)\sqrt{1-\rho^2})}
\\ &=
\frac{(r_2+\rho)(1-r_1\rho-2\rho^2)+(r_1+2\rho)(1-\rho^2)+i(1-r_1r_2-2(r_1+r_2)\rho-4\rho^2)\sqrt{1-\rho^2}}{(r_2+\rho)(1-r_1\rho-2\rho^2)+(r_1+2\rho)(1-\rho^2)-i(1-r_1r_2-2(r_1+r_2)\rho-4\rho^2)\sqrt{1-\rho^2}}.
\end{align*}
We then see that
$$
\tan \frac{d+\Delta}{2}=\frac{(1-r_1r_2-2(r_1+r_2)\rho-4\rho^2)\sqrt{1-\rho^2}}{(r_2+\rho)(1-r_1\rho-2\rho^2)+(r_1+2\rho)(1-\rho^2)}=\tan(\epsilon+\delta+\beta),
$$
where the last equality follows from \eqref{eq:defangles} and straightforward calculations.
The proof concludes by recalling the two following facts: 
\begin{enumerate}
\item For $\alpha=\frac{\epsilon+\delta-\pi}{\beta} \geqslant 1$ and for $\epsilon,\delta$ and $\beta\in(0,\pi)$, we have that $-\pi<2\beta-\pi\leqslant\epsilon+\delta+\beta-2\pi<\pi$.
\item By \eqref{eq:defangles}, $\sin(\epsilon+\delta+\beta)$ has the same sign as that $(r_1+2\rho)(r_2+2\rho)-1$, where
$(r_1+2\rho)(r_2+2\rho)-1 =\sin(\epsilon+\delta+\beta) \frac{\sin(\beta)}{\sin(\epsilon)\sin(\delta)}$. \
\end{enumerate}
 %Lemma~\ref{lem:condgeom} is also used for the last formula. 
%As it is only a technical issue we omit the details.
\end{proof}
We now prepare to state Lemma \ref{lem:ytilde} below.  For $1-(r_1+2\rho)(r_2+2\rho)\neq 0$, let us define
\begin{equation}
\widetilde{y}:=2\frac{\mu_2-\mu_1(r_2+2\rho)}{(r_1+2\rho)(r_2+2\rho)-1}
=2\mu_1 \frac{\sin(\beta+\delta-\theta)\sin(\epsilon)}{\sin(\beta-\theta)\sin(\epsilon+\delta+\beta)},
\label{eq:defytilde}
\end{equation}
where the last equality holds by \eqref{eq:defangles}.
\begin{lemma}
If $\widetilde{y}-y^+\leqslant 0$ or if $1-(r_1+2\rho)(r_2+2\rho)=0$ then 
$$( G(x)=1 \text{ and } x\in\mathcal{H} ) \Leftrightarrow x=X^\pm(y^+),$$
and thus $d+\Delta\in(-2\pi,2\pi)$.
If $\widetilde{y}-y^+>0$ then 
$$(G(x)=1 \text{ and } x\in\mathcal{H} ) \Leftrightarrow ( x=X^\pm(y^+) \text{ or }x=X^\pm(\widetilde{y}) ),$$
and thus $d+\Delta\in(-4\pi,4\pi)$.
\label{lem:ytilde}
\end{lemma}
\begin{proof}
Assume that $x\in\mathcal{H}$, where $x=a+ib$ for $a,b\in\mathbb{R}$ and $y=Y^\pm (x)$. Then by \eqref{eq:Gdefbrownien}, $G(x)=1$ is equivalent to $\Im ( k_1(a+ib,y) k_2(a-ib,y))=0$. Straightforward calculations yield
$$
\Im ( k_1(a+ib,y) k_2(a-ib,y))=\frac{b}{4}\left[\frac{y}{2}((r_1+2\rho)(r_2+2\rho)-1)-2\mu_2+2\mu_1(r_2+2\rho)\right],
$$
from which we may obtain that $G(x)=1$ is equivalent to $b=0$ or to $y=\widetilde{y}$. We conclude the proof by noting that 
\begin{enumerate}
\item $b=0$ and $x\in\mathcal{H}$ together imply that $x=X^\pm(y^+)$, the latter being the only real point of the hyperbola. 
\item By the definition of \eqref{eq:defHyp}, $x\in\mathcal{H}$ and $y=\widetilde{y}$ imply that $\widetilde{y}\in[y^+,\infty)$.% for $Y^\pm(x)\in[y^+,\infty)$.
\end{enumerate}
\end{proof}
We continue with Lemma \ref{lem:techny} below.
\begin{lemma}
Assume that $2\delta-\theta+\beta>2\pi$. Then $\widetilde{y}>y^+$ is equivalent to $\epsilon+\delta+\beta<2\pi$.
\label{lem:techny}
\end{lemma}
\begin{proof}
We first note that $2\delta-\theta+\beta>2\pi$ implies that $\pi<\delta-\theta+\beta<2\pi$, and thus that $\sin(\delta-\theta+\beta)<0$. Recall that we have previously seen that the conditions in \eqref{eq:conditionRmatrix} are equivalent to $\alpha\geqslant 1$ and $\delta>\beta$, $\epsilon>\beta$, and thus that $\pi<\epsilon+\delta+\beta<3\pi$. We employ the following steps to conclude the proof:
\begin{enumerate}
\item Assume that $\widetilde{y}>y^+$. Then for $y^+>0$, we have that $\widetilde{y}>0$. Then by \eqref{eq:defytilde} we have that $\sin(\epsilon+\delta+\beta)<0$ and thus $\epsilon+\delta+\beta<2\pi$. 
\item We %proceed from the step in (1) by 
now assume that $\epsilon+\delta+\beta<2\pi$. Hence $\sin(\epsilon+\delta+\beta)<0$. By hypothesis we have $\beta<\epsilon <2\pi-\beta-\delta$. Using \eqref{eq:defytilde}, we may easily see that $\epsilon\mapsto \widetilde{y}$ is increasing for $\beta<\epsilon <2\pi-\beta-\delta$. Replacing $\epsilon$ by $\beta$ in \eqref{eq:defytilde}, we may deduce that
$$\widetilde{y}> y_\delta:=2\mu_1 \frac{\sin(\beta+\delta-\theta)\sin(\beta)}{\sin(\beta-\theta)\sin(2\beta+\delta)}.$$
By hypothesis, we have that $\pi+\frac{\theta-\beta}{2}<\delta<2\pi-2\beta$. Note that $\delta\mapsto y_\delta$ is increasing in this interval. We then see that
$$\widetilde{y}> y_{\rems{\text{inf}}}:=2\mu_1 \frac{\sin(\beta+\pi+\frac{\theta-\beta}{2}-\theta)\sin(\beta)}{\sin(\beta-\theta)\sin(2\beta+\pi+\frac{\theta-\beta}{2})}.$$
Employing \eqref{eq:y+angle} and performing straightforward calculations, 
we obtain
$$
\widetilde{y}-y^+> y_{\rems{\text{inf}}}-y^+
=\mu_1
\frac{-2\sin(\frac{\beta-\theta}{2})\sin^2(\frac{\beta+\theta}{2})}{\sin(\beta-\theta)\sin(\epsilon+\delta+\beta)\sin(\beta)}
>0.
$$
\end{enumerate}
\end{proof}
Before stating the main \rems{l}emma of this section, we introduce the following indicator %random 
variable~$\chi$, which is associated with the results of Lemma~\ref{lem:poleBrownian} and Lemma~\ref{lem:condgeom}. 
\begin{equation}
\chi:=
\begin{cases}
-1 & \text{if } 2\delta -\theta +\beta> 2\pi 
%k_1(X^\pm(y^+),y^+) < 0 
\Leftrightarrow x_1 \text{ is a pole of } \psi_1 \text{ in } \mathcal{G},
\\
0 & \text{if } 
2\delta -\theta +\beta \leqslant 2\pi 
%k_1(X^\pm(y^+),y^+) \geqslant 0 
\Leftrightarrow \psi_1 \text{ has no pole but } 0 \text{ in } \mathcal{G}.
\end{cases}
\label{eq:chi}
\end{equation}
\begin{lemma}[Index]
The index $\kappa$ satisfies
$$
\kappa:=
\begin{cases}
\chi & \text{if } \epsilon+\delta +\beta\geqslant 2\pi ,
\\
\chi-1 & \text{if } 
\epsilon+\delta +\beta< 2\pi  .
\end{cases}
$$
The value of the index appears below in Table~\ref{table}.
\label{lem:index}
\end{lemma}
%\begin{center}

\renewcommand{\arraystretch}{1.5}
\begin{table}[h!]
\begin{tabular}{|c|c|c|}
\hline 
 & $\epsilon+\delta+\beta\geqslant 2\pi$ & $\epsilon+\delta+\beta< 2\pi$ \\ 
\hline 
$2\delta-\theta+\beta>2\pi$ & $\kappa=-1$ & $\kappa=-2$ \\ 
\hline 
$2\delta-\theta+\beta\leqslant2\pi$ & $\kappa=0$ & $\kappa=-1$ \\ 
\hline 
\end{tabular}
\medskip
\caption{Value of the index $\kappa$.}
\label{table}
\end{table} 
%\end{center}
\rems{
\begin{remark}[Index and argument principle] Notice that the index can take the values $0$, $-1$ and $-2$ while in \cite[Lemma 14]{franceschi_2019} the index takes only the values $0$ and $-1$. The difference comes from the fact that $\psi_1$ can have two distinct poles while in \cite{franceschi_2019} the Laplace transform has at most one simple pole. The index is deeply connected to number of zeros and poles of $\psi_1$. In the case of a closed curve, the argument principle implies that the index is equal to the number of zeros minus the number of poles counted with multiplicity of the function of the BVP. See \cite[Lemma 6.9]{fomichov2020probability} which presents a case where the boundary of the 
BVP is a circle. In our case, the boundary is an (unbounded) hyperbola and $\psi_1$ is not meromorphic at infinity, therefore we cannot apply directly the argument 
principle and the index $\kappa$ is not always equal to the opposite of the number of poles $\chi$.
\end{remark}}
\begin{proof}
The proof proceeds with two separate cases.\\\\
%The proof follows immediately by combining Lemma~\ref{lem:poleBrownian}, Lemma~\ref{lem:condgeom}, and \cite[Lemma 14]{franceschi_2019} (which studies a similar quantity).
\underline{Case I: $\widetilde{y}-y^+ \leqslant 0$.} In this case, by Lemma~\ref{lem:ytilde}, we have that ${d+\Delta}\in (-2\pi,2\pi)$ and that $G(x)\neq 1$ for all $x\in\mathcal{H}$ such that $x\neq X^\pm(y^+)$. Then $\kappa=0$ or $-1$ depending on the sign of $d+\Delta$. This sign is given by 
the sign of $\arg G(x)$ when $x\in\mathcal{H}^+$ and $x\to X^\pm(y^+)$. Note that $x=a+ib\in\mathcal{H}^+$ and $y=Y^+(x)$. We then compute
$$
k_1(a+ib,y)k_2(\overline{a+ib},y)=k_1(a,y)k_2(a,y) +\frac{b^2}{4}(r_2+2\rho) - i \frac{b}{4}(1-(r_1+2\rho)(r_2+2\rho))(y-\widetilde{y}).
$$
Figure~\ref{fignum} represents the curve $\mathcal{C}:=\{  \frac{k_1}{k_2} (x,Y^+(x)) : x\in\mathcal{H} \}$. It is useful to remark that $\arg \frac{k_1}{k_2}(x,Y^+(x)) =\arg k_1(x,Y^+(x))/k_2(\overline{x},Y^+(x))$.
We may thus deduce that 
\begin{align*}
\text{sgn} \arg G(x) &= \text{sgn} \arg (k_1(a+ib,y){k_2(a-ib,y)})
%\\ &=-\text{sgn} \arg \left( k_1(a,y)k_2(a,y) +\frac{b}{4}(r_2+2\rho) + i \frac{b}{2}(1-(r_1+2\rho)(r_2+2\rho))(y-\widetilde{y}) \right)
\\ &=\text{sgn} \frac{-b(1-(r_1+2\rho)(r_2+2\rho))(y-\widetilde{y})}{k_1(a,y)k_2(a,y) +\frac{b^2}{4}(r_2+2\rho)}.
\end{align*}
For $x\in\mathcal{H}^+$, we have $k_2(X^\pm(y^+),y^+)>0$, $b>0$. When $x\to X^\pm(y^+)$, we have that $b\to 0$ and $a\to X^\pm(y^+)$. Thus for $x\in\mathcal{H}^+$ and $x\to X^\pm(y^+)$,
\begin{align*}
\text{sgn} \arg G(x) &= 
- \text{sgn} (k_1(X^\pm(y^+),y^+)(1-(r_1+2\rho)(r_2+2\rho))(y-\widetilde{y}) )
\\ &= -\text{sgn} (2\delta-\theta+\beta-2\pi)\text{sgn}(\epsilon+\delta+\beta-2\pi),
\end{align*}
where the last equality comes from Lemmas~\ref{lem:condgeom} and \ref{lem:techniqueangle}, as well as from the fact that in this case $y-\widetilde{y}>0$ for $y> y^+$. This allows us to conclude the following
\begin{itemize}
\item If $\epsilon+\delta+\beta \geqslant 2\pi$ and $2\delta-\theta+\beta>2\pi$, then for $x\in\mathcal{H}^+$ and $x\to X^\pm(y^+)$, the sign of $ \arg G(x) $ is negative. We may thus deduce that $\kappa=-1$, see Figure~\ref{subfig1}.
\item If $\epsilon+\delta+\beta \geqslant 2\pi$ and $2\delta-\theta+\beta\leqslant 2\pi$, then for $x\in\mathcal{H}^+$ and $x\to X^\pm(y^+)$, the sign of $ \arg G(x) $ is positive. We may thus deduce that $\kappa=0$, see Figure~\ref{subfig2}. 
\item If $\epsilon+\delta+\beta < 2\pi$ and $2\delta-\theta+\beta\leqslant 2\pi$, then for $x\in\mathcal{H}^+$ and $x\to X^\pm(y^+)$,  the sign of $ \arg G(x) $ is positive. We may thus deduce that $\kappa=-1$, see Figure~\ref{subfig3}.
\end{itemize}
We pause to note that by Lemma~\ref{lem:techny} it is not possible to have $\epsilon+\delta+\beta < 2\pi$ and $2\delta-\theta+\beta>2\pi$. This is because we have assumed $\widetilde{y}\leqslant y^+$.\\

%Figure~\ref{fig:curve1} illustrate these cases.
\noindent \underline{Case II: $\widetilde{y}-y^+ > 0$.} In this case by, Lemma~\ref{lem:ytilde} we have that ${d+\Delta}\in (-4\pi,4\pi)$ and $(G(x)=1 \text{ and } x\in\mathcal{H} ) \Leftrightarrow ( x=X^\pm(y^+) \text{ or }x=X^\pm(\widetilde{y}) ) $. Then $\kappa \in\{ -2,-1,0,1 \}$. To obtain the value of the index we study the curve $\mathcal{C}:=\{  \frac{k_1}{k_2} (x,Y^+(x)) : x\in\mathcal{H} \}$. %Figure~\ref{fignum} represents a numerical plot of this curve in the different cases. 
By straightforward calculations we see that $\widetilde{A}:=\frac{k_1}{k_2}(X^\pm(\widetilde{y}),\widetilde{y})$ is positive. The study of the sign of the real and the imaginary parts of $ \frac{k_1}{k_2}(x,Y^+(x))$ for $x\in\mathcal{H}^+$ and $x\to X^\pm(y^+)$ gives the value of $\kappa$. Following the same logic as that of Case I above, we see that \remst{that }the real part of $ \frac{k_1}{k_2}(x,Y^+(x))$ for $x\in\mathcal{H}^+$ and $x\to X^\pm(y^+)$ has the same sign as $-(2\delta-\theta+\beta-2\pi)$. Further, the imaginary part \rems{has} the same sign that $-(\epsilon+\delta+\beta-2\pi)$. We may then conclude as follows:
\begin{itemize}
\item If $\epsilon+\delta+\beta < 2\pi$ and $2\delta-\theta+\beta>2\pi$, $\kappa=-2$, see Figure~\ref{subfigA}.
\item If $\epsilon+\delta+\beta \geqslant 2\pi$ and $2\delta-\theta+\beta\leqslant 2\pi$, $\kappa=0$, see Figure~\ref{subfigB}.
\item If $\epsilon+\delta+\beta < 2\pi$ and $2\delta-\theta+\beta\leqslant 2\pi$, $\kappa=-1$, see Figure~\ref{subfigC}.
\end{itemize}
\item Note that by Lemma~\ref{lem:techny} it is not possible to have $\epsilon+\delta+\beta \geqslant 2\pi$ and $2\delta-\theta+\beta>2\pi$. This is because we have assumed that $\widetilde{y}> y^+$. 

\begin{figure}[p]
    \centering
    \begin{subfigure}[b]{.3\linewidth}
        \centering
        \includegraphics[trim=0.5cm 0.5cm 0.5cm 0.5cm ,clip=true,width=\textwidth]{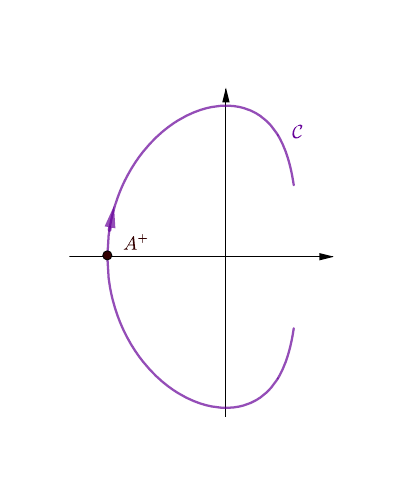}
        \caption{If $\epsilon+\delta+\beta \geqslant 2\pi$ and $2\delta-\theta+\beta>2\pi$, then $\kappa=-1$.}
        \label{subfig1}
    \end{subfigure}
    \begin{subfigure}[b]{.3\linewidth}
        \centering
        \includegraphics[trim=0.5cm 0.5cm 0.5cm 0.5cm ,clip=true, width=\textwidth]{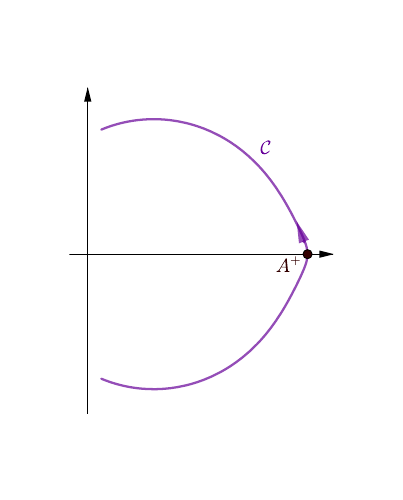}
    \caption{If $\epsilon+\delta+\beta \geqslant 2\pi$ and $2\delta-\theta+\beta\leqslant 2\pi$, then $\kappa=0$.}
    \label{subfig2}
    \end{subfigure}
     \begin{subfigure}[b]{.3\linewidth}
        \centering
        \includegraphics[trim=0.5cm 0.5cm 0.5cm 0.5cm ,clip=true,width=\textwidth]{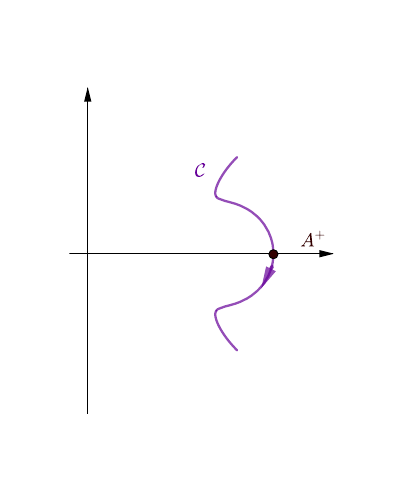}
    \caption{If $\epsilon+\delta+\beta < 2\pi$ and $2\delta-\theta+\beta\leqslant 2\pi$, then $\kappa=-1$.}
    \label{subfig3}
    \end{subfigure}
    \caption{When $\widetilde{y}-y^+ \leqslant 0$: a plot of the curve $\mathcal{C}:=\{\frac{k_1}{k_2}(x,Y^+(x)) : x\in\mathcal{H} \}$ and the point $A^+:= \frac{k_1}{k_2}(X^+(y^+),y^+) $.}
    \label{fignum}
\end{figure} 

\begin{figure}[p]
    \centering
    \begin{subfigure}[b]{.45\linewidth}
        \centering
        \includegraphics[trim=0.5cm 0.5cm 0.5cm 0.5cm ,clip=true,width=\textwidth]{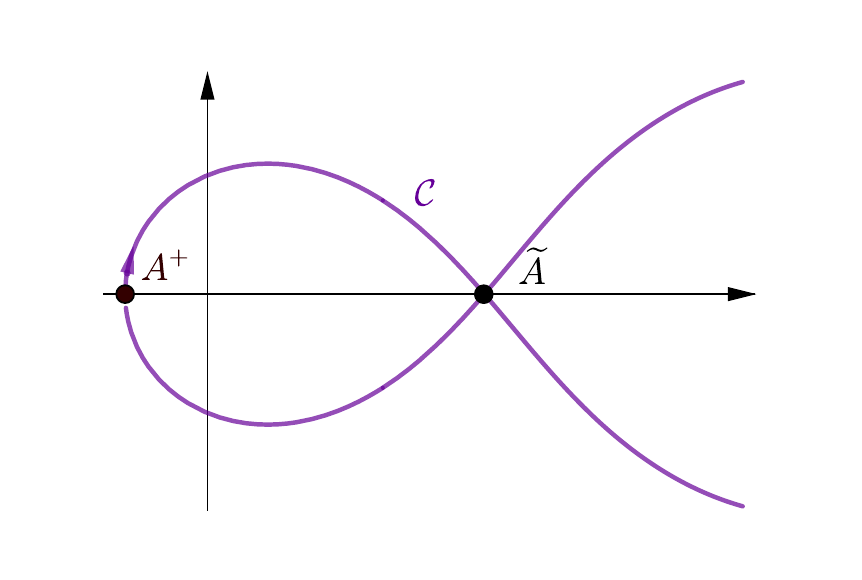}
        \caption{If $\epsilon+\delta+\beta < 2\pi$ and $2\delta-\theta+\beta>2\pi$,\\ then $\kappa=-2$.}
        \label{subfigA}
    \end{subfigure}
    \\
    \begin{subfigure}[b]{0.45\linewidth}
        \centering
        \includegraphics[trim=0.5cm 0.5cm 0.5cm 0.5cm ,clip=true,width=\textwidth]{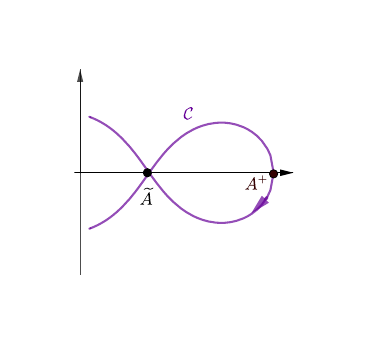}
    \caption{If $\epsilon+\delta+\beta \geqslant 2\pi$ and $2\delta-\theta+\beta\leqslant 2\pi$,\\ then $\kappa=0$.}
    \label{subfigB}
    \end{subfigure}
     \begin{subfigure}[b]{.45\linewidth}
        \centering
        \includegraphics[trim=0.5cm 0.5cm 0.5cm 0.5cm ,clip=true,width=\textwidth]{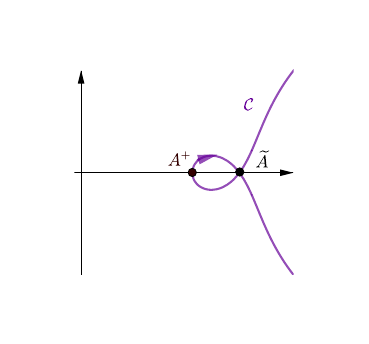}
    \caption{If $\epsilon+\delta+\beta < 2\pi$ and $2\delta-\theta+\beta\leqslant 2\pi$, \\ then $\kappa=-1$.}
    \label{subfigC}
    \end{subfigure}
    \caption{When $\widetilde{y}-y^+ > 0$: a plot of the curve $\mathcal{C}:=\{ \frac{k_1}{k_2}(x,Y^+(x))  : x\in\mathcal{H} \}$, the point $A^+:=\frac{k_1}{k_2}(X^+(y^+),y^+), $ and the point $\widetilde{A}:=\frac{k_1}{k_2}(X^+(\widetilde{y}),\widetilde{y}) $.}
    \label{figlettre}
\end{figure} 
\end{proof}

We now state a technical lemma which shall be invoked in Section~\ref{sec:asympt0}.
\begin{lemma}
The following equality holds
$$
\left(-\frac{d+\Delta}{2\pi} +\chi  -1\right)\frac{\pi}{\beta} 
=
-\alpha-1.
$$
\label{lem:chiDelta}
\end{lemma}
\begin{proof}
First, we recall by Lemma~\ref{lem:techniqueangle} that
$$
\tan \frac{d+\Delta}{2} = \tan (\epsilon+\delta+\beta).
$$
For $\alpha=\frac{\epsilon+\delta-\pi}{\beta} \geqslant 1$ and $\epsilon,\delta$ and $\beta\in(0,\pi)$, 
$$
2\beta-\pi\leqslant\epsilon+\delta+\beta-2\pi<\pi.
$$
Further, recall that by definition, $\kappa=\lfloor \frac{d+\Delta}{2\pi} \rfloor $.\\
\indent We now consider two cases for the value of $\epsilon+\delta+\beta-2\pi$. The first case considers $\epsilon+\delta+\beta-2\pi\geqslant 0$. In this case, 
$$ \frac{d+\Delta}{2} =
\begin{cases}
\epsilon+\delta+\beta-2\pi & \text{if } \frac{d+\Delta}{2\pi}\geqslant 0 \text{ i.e. } \kappa=0,
\\ 
\epsilon+\delta+\beta-3\pi & \text{if } \frac{d+\Delta}{2\pi}<0 \text{ i.e. } \kappa=-1.
\end{cases}
$$ By Lemma~\ref{lem:index}, we have $\kappa=\chi$. We may thus deduce that
$$
\frac{d+\Delta}{2} = \epsilon+\delta+\beta +(\chi-2)\pi.
$$
The second case considers $\epsilon+\delta+\beta-2\pi<0$. In this case,
$$ \frac{d+\Delta}{2} =
\begin{cases}
\epsilon+\delta+\beta-2\pi & \text{if } -\pi\leqslant\frac{d+\Delta}{2\pi}<0 \text{ i.e. } \kappa=-1,
\\ 
\epsilon+\delta+\beta-3\pi & \text{if } -2\pi\leqslant\frac{d+\Delta}{2\pi}<-\pi \text{ i.e. } \kappa=-2.
\end{cases}
$$
By Lemma~\ref{lem:index}, we have $\kappa=\chi-1$. We may thus deduce that
$$
\frac{d+\Delta}{2} = \epsilon+\delta+\beta +(\chi-2)\pi.
$$
Thus, in both cases we have
\begin{align*}
\left(-\frac{d+\Delta}{2\pi} +\chi  -1\right)\frac{\pi}{\beta} 
&= \left(-\epsilon-\delta-\beta-(\chi-2)\pi +\chi \pi -\pi\right)\frac{1}{\beta} 
%\\ &=
=
-\alpha-1.
\end{align*}
This concludes the proof.
\end{proof}

\subsection{Solution of the BVP}

\rems{The following theorem gives an explicit integral formula for the Laplace transform of the escape probability $\psi_1$.}

\begin{theorem}[Explicit expression for $\psi_1$]
\label{thm:explicitF1Brown}
The Laplace transform $\psi_1$ is given for $x\in\mcG$ by
\begin{equation}
\label{eq:explicitexpressionBrownian}
\psi_1(x)= \dfrac{w'(0)}{w(x)-w(0)} \left(\dfrac{w(0)-w(x_1)}{w(x)-w(x_1)}\right)^{-\chi} \exp\left(\dfrac{1}{2i\pi} \int_{\mcH^+} \log G(t)
\left[ \dfrac{w'(t)}{w(t)-w(x)}
-\dfrac{w'(t)}{w(t)-w(0)} \right]
\D t\right) ,
\end{equation}
where $x_1$ is defined in \eqref{eq:x1}, $G$ is defined~\eqref{eq:Gdefbrownien}, $w$ is defined~\eqref{eq:wdefBrown}, $\chi$ is defined in~\eqref{eq:chi} and $\mathcal{H}$ is defined in \eqref{eq:hyperbole}.
\end{theorem}
\rems{\begin{remark}
Let us now give some remarks about Theorem~\ref{thm:explicitF1Brown}.
\begin{itemize}
\item The poles $0$ and $x_1$ found in Lemma~\ref{lem:poleBrownian} can be easily visualized in the formula of Theorem~\ref{thm:explicitF1Brown}. The indicator variable $\chi$ defined in~\eqref{eq:chi} indicates clearly on the formula if the pole $x_1$ is in $\mathcal{G}$ or not.
\item A symmetrical result holds for $\psi_2$. Using the functional equation~\eqref{eq:functionalequation} we obtain an explicit formula for $\psi$. By inverting this Laplace transform we obtain the escape probability which is the main motivation of our work. But such an inversion is not easy neither very explicit except in some special cases. 
\item But it is still possible to deduce some concrete results from the integral formula obtained in Theorem~\ref{thm:explicitF1Brown}. In Section~\ref{sec:asympt0} we derive thanks to this explicit expression a very explicit and simple expression for the asymptotics of the escape probability at the origin. 
\item It can also be used to show some differential properties of the Laplace transform. More precisely, similarly to \cite[Thm 2.3, \S 9.1]{BM2021stationary} we can show that $\psi_1$ is differentially algebraic if $\beta\in\pi \mathbb{Q}$. Such results on the algebraic nature of a generating function are very classical in analytic combinatorics to obtain concrete results. When $\psi_1$ is differentially algebraic, it satisfies a differential equation from which it is possible to deduce
a polynomial recurrence relation for the moments of the escape/absorption probability. See \cite[\S 6.3]{BM2021stationary} which gives an explicit example for the SRBM stationary distribution in the recurrent case.
\item The methods and techniques employed to prove this theorem are inspired by the one used to study random walks in the quarter plane \cite{fayolle_random_2017}.
\end{itemize}
\end{remark}}

\begin{proof}
Let $$\widetilde{\psi}_1 (y):=\frac{(y- W(x_1))^{-\chi}}{(y-1)^{1+\kappa-\chi}} \psi_1 \circ W^{-1} (y).$$
Proposition~\ref{prop:BVPBrownian}, Lemma~\ref{lem:gluingwBrownien} and Lemma~\ref{lem:index} together imply that
\begin{itemize}
\item $\widetilde{\psi}_1$ is analytic on $\mathbb{C}\setminus [0,1]$.
\item $\widetilde{\psi}_1 (y) \sim_{\infty} c y^{-\kappa}$ for some constant $c$.
\item $\widetilde{\psi}_1 (1) =0$.
\item For $y\in [0,1]$, $\widetilde{\psi}_1$ satisfies the boundary condition
$$
\widetilde{\psi}_1^+ (y) = \widetilde{G} (y)  \widetilde{\psi}_1^-  (y),
$$
where $\widetilde{\psi}_1^+ (y)$ is the left limit and $\widetilde{\psi}_1^- (y)$ is the right limit of $\widetilde{\psi}_1$ on $[0,1]$, $(W^{-1})^-$ is the right limit of $W^{-1}$ on $[0,1]$, and $\widetilde{G} (y) = G \circ (W^{-1})^- (y)$.
\end{itemize}
We now define
$$
\widetilde{S}(y):= (y-1)^{-\kappa} \exp \left( \frac{1}{2i\pi} \int_0^1 \frac{\log \widetilde{G} (u)}{u-y} \right).
$$
Following the classical boundary theory results in \cite[(5.2.24) and
Theorem 5.2.3]{fayolle_random_2017}, the above function is analytic and does not cancel on $\mathbb{C}\setminus [0,1]$ and is such that $\widetilde{S} (y) \sim_{\infty} c' y^{-\kappa}$ for some constant $c'$. Furthermore, for $y\in[0,1]$, it satisfies the boundary condition
$$
\widetilde{S}^+ (y) = \widetilde{G} (y)  \widetilde{S}^-  (y),
$$
where $\widetilde{S}^+ (y)$ is the left limit and $\widetilde{S}^- (y)$ is the right limit of $\widetilde{S}$ on $[0,1]$. 
By the properties of $\widetilde{\psi}_1$ and $\widetilde{S}$ detailed above, the function
$
\widetilde{\psi}_1 / \widetilde{S}
$
is analytic on $\mathbb{C}$ and bounded at infinity. Therefore there must exist a constant $C$ such that
$$
\widetilde{\psi}_1 (y)=  C\widetilde{S}(y).
$$
%The constant $C$ denote a generic constant which can vary from one line to another.
Invoking the definition of $\widetilde{\psi}_1$, we have that
\begin{equation}\label{BVPeq}
\frac{(W(x)- W(x_1))^{-\chi}}{(W(x)-1)^{1+\kappa-\chi}} \psi_1 (x) = C (W(x)-1)^{-\kappa} \exp \left( \frac{1}{2i\pi} \int_0^1 \frac{\log \widetilde{G} (u)}{u-W(x)} \right).
\end{equation}
Noting that
$$
W(x)-1=\frac{w(0)-w(X^\pm(y^+))}{w(x)-w(0)}
\quad \text{and} \quad
W(x)-W(x_1)=\frac{w(x)-w(x_1)}{w(x)-w(0)}\frac{w(X^\pm(y^+))-w(0)}{w(x_1)-w(0)},
$$
and making a change of variable $u=w(t)$ in the integral in \eqref{BVPeq}, we obtain for some constant $C'$
$$
\psi_1(x)=C' \left(\dfrac{1}{w(x)-w(0)}\right) \left(\dfrac{1}{w(x)-w(x_1)}\right)^{-\chi} \exp\left(\dfrac{1}{2i\pi}  \int_{\mcH^+} \log G(t)
\dfrac{w'(t)}{w(t)-w(x)}
\D t\right).
$$
The final value theorem for the Laplace transform gives $$\lim_{x\to 0} x \psi_1(x) =  \lim_{u\to\infty}\mathbb{P}_{(u,0)} [T=\infty] =1.$$
\noindent This enables us to compute the constant 
$$C'=w'(0) \left({w(0)-w(x_1)}\right)^{-\chi} \exp\left(\dfrac{-1}{2i\pi}  \int_{\mcH^+} \log G(t)
\dfrac{w'(t)}{w(t)-w(0)}
\D t\right)\rems{,}$$ 
which gives us~\eqref{eq:explicitexpressionBrownian}, completing the proof.
\end{proof}

\section{Asymptotics of the escape probability at the origin}
\label{sec:asympt0}

In this section we use the explicit expression in Theorem~\ref{thm:explicitF1Brown} to obtain the asymptotics of the escape probability at the origin. We begin with computing the asymptotics of $\psi_1$ at infinity. 
\begin{lemma}[Asymptotics of $\psi_1$]
Let $\alpha$ be defined as in \eqref{eq:alpha}. For ease of notation, allow $C$ to be a constant which may change from one line to the next. For some positive constant $C$,
\begin{equation*}
\psi_1(x) \underset{x\to\infty}{\sim} C x^{-\alpha -1}.
\end{equation*}
A \rems{symmetrical} result holds for $\psi_2$. That is, for some positive constant $C$,
\begin{equation*}
\psi_2(y) \underset{y\to\infty}{\sim} C y^{-\alpha -1}.
\end{equation*}
\label{lem:asymptpsi1infty}
\end{lemma}
\begin{proof}

This proof follows the same steps as those of \cite[Prop 19]{franceschi_2019}. The key is to invoke \cite[(5.2.20)]{fayolle_random_2017}, which states that
$$
\exp \left( \frac{1}{2i\pi} \int_0^1 \frac{\log \widetilde{G} (u)}{u-y} \right)
\underset{y\to 1}{\sim}  C (y-1)^{\frac{d+\Delta}{2\pi}}
. $$
Recall that $w(x)\underset{x\to\infty}{\sim} C x^{\frac{\pi}{\beta}}$ and that $W(x)-1\underset{x\to\infty}{\sim} C x^{-\frac{\pi}{\beta}}$. The explicit expressions of $\psi_1$ obtained in \eqref{eq:explicitexpressionBrownian} and in \eqref{BVPeq}
imply that
\begin{equation*}
\psi_1(x) \underset{x\to\infty}{\sim} C x^{(-\frac{d+\Delta}{2\pi} +\chi -1)\frac{\pi}{\beta}}.
\end{equation*}
The proof concludes by invoking Lemma~\ref{lem:chiDelta}, which states that $\left(-\frac{d+\Delta}{2\pi} +\chi -1\right)\frac{\pi}{\beta}=-\alpha-1$. 
\end{proof}

\begin{lemma}[Asymptotics of $\psi$]
Let $\alpha$ defined as in \eqref{eq:alpha}. For $t\in[0,\frac{\pi}{2}]$ and some positive constant $C_t$, 
\begin{equation*}
\psi(r \cos t,r \sin t) \underset{r\to\infty}{\sim} C_t r^{-\alpha -2}.
\end{equation*}
\label{lem:asymppsiinfty}
\end{lemma}
\begin{proof}
The result is immediate from the functional equation~\eqref{eq:functionalequation} and Lemma~\ref{lem:asymptpsi1infty}.
\end{proof}

\begin{proposition}[Asymptotics at the origin]
For positive constants $c_0$ and $c_1$ we have the following asymptotics
$$
\mathbb{P}_{(u,0)}[T=\infty] \underset{u\to 0}{\sim} c_0 u^\alpha
\quad \text{and} \quad
\mathbb{P}_{(0,v)}[T=\infty] \underset{v\to 0}{\sim} c_1 v^\alpha.
$$
\label{prop:asymptaxeorigin}
\end{proposition}
\begin{proof}
The result follows by combining the asymptotic results of $\psi_1$ and $\psi_2$ at infinity that we computed in Lemma~\ref{lem:asymptpsi1infty} with the reciprocal of the result in \cite[Thm 33.3]{doetsch_introduction_1974}\footnote{\citet[Thm 33.3]{doetsch_introduction_1974} establishes that if for some constant $a$ a function is equivalent to $u^a$ at $0$, then at infinity, its Laplace transform is equivalent (up to a multiplicative constant) to $x^{-a-1}$.}. We begin by denoting $g(u):=\mathbb{P}_{(u,0)}[T=\infty]$. Then, by definition,  $\psi_1(x)=\int_0^\infty e^{-xu}g(u)\mathrm{d}u$. As $\psi_1(x)$ has no singularities greater than $0$, for every $A>0$, the inverse Laplace transform gives
$$
g(u)=\frac{1}{2i\pi}\int_{A-i\infty}^{A+i\infty} e^{ux}\psi_1(x)\mathrm{d}x.
$$
By Lemma~\ref{lem:asymptpsi1infty}, we have $\psi_1(x)=\frac{C+\eta(x)}{x^{\alpha+1}}$, where $\eta$ is a function such that $\lim_\infty \eta =0$. Recalling that the Laplace transform of $u^\alpha/\Gamma(\alpha+1) $ is $x^{-\alpha-1}$ and performing a change of variables $s=ux$, we obtain
\begin{align*}
g(u) &= \frac{1}{2i\pi}\int_{A-i\infty}^{A+i\infty} e^{ux}\frac{C+\eta(x)}{x^{\alpha+1}}\mathrm{d}x
%\\&=\frac{u^\alpha}{2i\pi} \int_{Au-i%\infty}^{Au+i\infty} e^{s}\frac{C+\eta(s/%u)}{s^{\alpha+1}}\mathrm{d}s
\\ &=
 u^\alpha\left( \frac{C}{\Gamma(\alpha+1)}+\frac{1}{2i\pi} \int_{Au-i\infty}^{Au+i\infty} e^{s}\frac{\eta(s/u)}{s^{\alpha+1}}\mathrm{d}s \right).
\end{align*}
It remains to show that the last integral tends to $0$ when $u\to 0$. To do so, consider $\epsilon>0$ arbitrarily small. Then there exists $B>0$ sufficiently large such that $\eta(x) < \epsilon$ for all $|x|>B$. For all $u$ such that $u<1/B$, let us consider $A:=1/u$. This gives 
$$
\left| 
\frac{1}{2i\pi} \int_{Au-i\infty}^{Au+i\infty} e^{s}\frac{\eta(s/u)}{s^{\alpha+1}}\mathrm{d}s
\right|
<
 \frac{\epsilon}{2i\pi} \int_{1-i\infty}^{1+i\infty} \frac{1}{s^{\alpha+1}}\mathrm{d}s,
$$
where the last integral converges for $\alpha\geqslant 1$. The proof concludes by letting $\epsilon$ tend towards $0$.
\end{proof}

\begin{theorem}[Asymptotics at the origin]
For $t\in(0,\frac{\pi}{2})$ and some positive constant $c_t$ we have
$$
\mathbb{P}_{(r \cos t,r \sin t)}[T=\infty] \underset{r\to 0}{\sim}  c_t r^\alpha.
$$
\end{theorem}
\begin{proof}
This proof follows directly from the asymptotics of the double Laplace transform $\psi$ computed in Lemma~\ref{lem:asymppsiinfty}. Recall the result used in the proof of Proposition~\ref{prop:asymptaxeorigin} linking the asymptotics of a function at $0$ to the asymptotics of its Laplace transform at infinity. The only necessary modification is to apply this result with a polar coordinate transformation. The desired asymptotics then follows with nearly identical calculations.
\end{proof}

\subsection*{Acknowledgments}

We thank L.C.G. Rogers for many inspiring conversations and motivational ideas, without which this paper would not have been written. We thank Kilian Raschel for various and useful discussions on this issue. 

%\section{Remarkable cases}
%\label{sec:remarkablecase}
%
%\subsection{Decoupling function}
%
%\subsection{Tutte's invariant}
%
%\subsection{Simplified formula}

\newpage

\bibliographystyle{apalike}
\bibliography{biblio2v2}
\end{document}